\pdfoutput=1
\pdfoutput=1
\pdfoutput=1
\pdfoutput=1
\pdfoutput=1
\documentclass[a4paper,12pt]{article}
\usepackage{calc}
\usepackage[all]{xy}
\usepackage[centertags]{amsmath}
\usepackage{latexsym}
\usepackage{amsfonts}
\usepackage{graphicx}
\usepackage{tikz}
\usepackage{cases}
\usepackage{amssymb}
\usepackage{amsthm}
\usepackage{color}
\usepackage{fancyhdr}
\usepackage[dvips]{epsfig}
\usepackage{newlfont}
\usepackage[latin1]{inputenc}
\usepackage[latin1]{inputenc}
\usepackage[english,french]{babel}
\usepackage{graphicx}
\usepackage{t1enc}
\usepackage[english,french]{babel}
\usepackage{fancybox}
\usepackage{graphicx}
\usepackage{t1enc}
\usepackage[french2]{minitoc}
\usepackage{mathrsfs}
\usepackage{amsfonts}
\usepackage{wasysym}
\usepackage{authblk}
\usepackage{hyperref}
\allowdisplaybreaks
\usepackage{float}

	\fancypagestyle{plain}{ \fancyhead{}

		\rhead{\textbf{}} \rfoot{\footnotesize{\textsf{\tiny }}}
		\cfoot{\footnotesize{\textbf{ \thepage}}}} \hfuzz2pt
	\newlength{\defbaselineskip}
	\numberwithin{equation}{section} 
	\newtheorem{theorem}{Theorem}[section]
	\newtheorem{corollary}[theorem]{Corollary}
	
	\newtheorem{lemma}[theorem]{Lemma}
	\newtheorem{proposition}[theorem]{Proposition}

	\newtheorem{remark}{Remark}[section]

	 \makeatletter
	\newcommand{\thechapterwords}
	{ \ifcase \thechapter\or 1\or 2\or 3\or 4\or 5\or
		6\or 7\or 8\or 9\or 10\or 11\fi}
	\def\thickhrulefill{\leavevmode \leaders \hrule height 2ex \hfill \kern \z@}
	\def\@makechapterhead#1{%
		\vspace*{15\p@}%
		{\parindent \z@ \centering \reset@font
			\thickhrulefill\quad
			\scshape  {\chapnumfont \@chapapp{}}{\chapnumfont \thechapterwords}
			\quad \thickhrulefill
			\par\nobreak
			\vspace*{15\p@}%
			\interlinepenalty\@M
			\hrule
			\vspace*{15\p@}%
			\huge {\bfseries  #1}\par\nobreak
			\par
			\vspace*{15\p@}%
			\hrule
			\vskip 15\p@
	}}
	\def\@makeschapterhead#1{%
		\vspace*{15\p@}%
		{\parindent \z@ \centering \reset@font
			\thickhrulefill
			\par\nobreak
			\vspace*{15\p@}%
			\interlinepenalty\@M
			\hrule
			\vspace*{15\p@}%
			\Huge \bfseries #1\par\nobreak
			\par
			\vspace*{15\p@}%
			\hrule
			\vskip 30\p@
	}}
	
	\DeclareFixedFont{\chapnumfont}{T1}{phv}{b}{n}{20pt}
	\DeclareFixedFont{\chapchapfont}{T1}{phv}{b}{n}{16pt}
	\DeclareFixedFont{\chaptitfont}{T1}{phv}{b}{n}{24.88pt}
	\def\@makechapterhead#1{%
		\vspace*{15\p@}%
		{\parindent \z@ \centering \reset@font
			\thickhrulefill\quad
			\scshape {\chaptitfont\color[rgb]{0.00,0.50,1.00}\@chapapp{}}
			{\chapnumfont \thechapterwords}
			\quad \thickhrulefill
			\par\nobreak
			\vspace*{15\p@}%
			\interlinepenalty\@M
			\hrule
			\vspace*{15\p@}%
			{\Large\bfseries #1}\par\nobreak
			\par
			\vspace*{15\p@}%
			\hrule
			\vskip 30\p@
	}}%
	
	\title{Bi-Lagrangian, 3-web and transverse Dirac structures  induced on the Whitney sum}
	\author[1]{Bertuel Tangue Ndawa}
	\author[2]{Ferdinand Ngakeu}
	\author[3]{Nasser Saipele Nansidi}
	\author[4]{Samuel Diangitukulu Ndimba}

	\affil[1]{
		University of Ngaoundere, Cameroon.}
	\affil[2]{
		University of Douala, Cameroon.}
	\affil[3]{
		University of Maroua, Cameroon.}
	\affil[4]{
		University of Kinshasa, Democratic Republic of the Congo.}

	\date{ }

\begin{document}
		
		\maketitle
		
		\selectlanguage{english}
		\begin{center}
			To Professor Michel Nguiffo Boyom.
		\end{center}
		
Let $(M,\omega,\mathcal{F}_1,\mathcal{F}_2)$ be a bi-Lagrangian
manifold. We construct, in several ways, induced bi-Lagrangian
structures on $TM$, $T^*M$, and the Whitney sum
$W=TM\oplus T^*M$, and investigate how their affinity relates to
that of the original structure. Some of these constructions also give
rise to $3$-webs on the corresponding total spaces. In addition, we
lift to these bundles the canonical action of the symplectomorphism
group on the space of bi-Lagrangian structures.
Given a distribution $\mathcal{D}\subseteq TM$, we characterize when
the subbundle $N^{*t}\mathcal{D}:=\mathcal{D}\oplus N^*\mathcal{D}$
of $W$ is a Dirac structure with respect to the $(H,\theta)$-twisted Current bracket. As a result, every
bi-Lagrangian structure canonically determines a pair of transverse
Dirac structures. Moreover, the associated bi-Lagrangian connection,
also known as the Hess connection, induces a linear connection on $W$
with respect to which both Dirac subbundles are parallel. This
establishes a natural link between bi-Lagrangian geometry and
Courant Dirac geometry.
 \vspace{0.025cm}\\
		
\vspace{0.025cm}\\
		
		
\textbf{Acknowledgment}: The first author gratefully acknowledges Professor Frank Loray for introducing him to the notion of 3-webs.

		\section{Introduction}
Let $(M,\omega)$ be a symplectic manifold; thus, $\omega$ is a closed and
nondegenerate differential $2$-form on $M$. A bi-Lagrangian structure on
$(M,\omega)$ is a pair $(\mathcal{F}_1,\mathcal{F}_2)$ of transverse
Lagrangian foliations. Equivalently, one may regard a bi-Lagrangian
structure on $M$ as a triple $(\omega,\mathcal{F}_1,\mathcal{F}_2),$
where $\omega$ is a symplectic form and $\mathcal{F}_1,\mathcal{F}_2$ are
transverse Lagrangian foliations with respect to $\omega$, see
\cite{GB1,1,Hess,TNB}. In either description,
$(M,\omega,\mathcal{F}_1,\mathcal{F}_2)$ is called a
bi-Lagrangian manifold.
Every bi-Lagrangian manifold carries a unique torsion-free symplectic
connection $\nabla$ that preserves both Lagrangian foliations. This
connection is known as the Hess connection, or equivalently the
bi-Lagrangian connection, see~\cite{Hess}. It is closely related to the Bott
connections associated with the two foliations and plays an important
role in geometric quantization(see~\cite{Hess}) and Koszul Vinberg cohomology, see~\cite{GB1}. Bi-Lagrangian geometry is equivalently described in terms of
para-K\"ahler geometry, see~\cite{1}.

In \cite{TNB}, the author introduced a lifting procedure for dynamics
on the space of bi-Lagrangian structures to the cotangent bundle
$T^*M$, under the assumption that $M$ is parallelizable. One of the objectives of the present work is to extend this construction, without the parallelizability assumption, to the Whitney sum $W:=TM\oplus T^*M$, which is the natural setting for generalized geometry.
To this end, we study two lifts of bi-Lagrangian structures to
$T^*M$, constructed using horizontal and vertical lifts, together with
the corresponding push-forward constructions on $TM$. Conversely, we
consider constructions on $TM$ based on complete, horizontal, and
vertical lifts and transport them to $T^*M$. We then combine these
constructions on the Whitney sum $W$. A central
observation is that lifting the underlying bi-Lagrangian structures is
sufficient to lift the dynamics induced by the push-forward action of
symplectomorphisms. Under suitable hypotheses, the lifted structures on
$TM$, $T^*M$, and $W$ admit local adapted affine coordinates. In
particular cases, these lifted structures also give rise to
$3$-webs.

 Given a distribution $\mathcal{D}\subseteq TM$, we consider
its conormal bundle $N^*\mathcal{D}:=
\left\{(x,\alpha_x)\in T^*M:\alpha_x(v)=0,\;v\in\mathcal{D}_x
\right\}$ 
and the associated subbundle
$N^{*t}\mathcal{D}:=\mathcal{D}\oplus N^*\mathcal{D}$.
 The latter has rank $\dim M$ and is
maximally isotropic with respect to the natural symmetric pairing
$\langle,\rangle$, see~\eqref{split<>_LED} and Remark~\ref{rem_LagW_ED}. In \cite{BCWZ}, the authors show that $N^{*t}\mathcal{D}$ is an $H$-twisted
Dirac structure if and only if $\mathcal{D}$ is involutive and
$H|_{\wedge^3\mathcal{D}}=0$, where $H$ is a closed $3$-form. Here, we study
the locally conformal twisted case; that is, the bracket twisted by a
$3$-form $H\in\Omega^3(M)$ and a $1$-form $\theta\in\Omega^1(M)$ satisfying
$dH=H\wedge\theta$ and $d\theta=0$. In this case, we also find that
$N^{*t}\mathcal{D}$ is a Dirac structure for the $(H,\theta)$-twisted Courant
bracket (defined in \eqref{C(theta,H)_DC}) if and only if $\mathcal{D}$ is
involutive and $H|_{\wedge^3\mathcal{D}}=0$. As in the closed-twister case, the
$(H,\theta)$-twisted Dirac structure $N^{*t}\mathcal{D}$ inherits a natural
Lie algebroid structure, whose anchor is the restriction of the projection
$\rho:TM\oplus T^*M\to TM$ and whose bracket is the restriction of the
$(H,\theta)$-twisted Courant bracket to $\Gamma(N^{*t}\mathcal{D})$.

Applying this construction to a bi-Lagrangian manifold $(M,\omega,\mathcal{F}_1,\mathcal{F}_2)$, we obtain
$L_i:=N^{*t}(\mathcal{F}_i)$, $i=1,2.$
Since $\mathcal{F}_1$ and $\mathcal{F}_2$ are involutive, $L_1$ and $L_2$ are Dirac structures for the untwisted Courant bracket, and for the $(H,\theta)$-twister Courant braket when $H|_{\wedge^3\mathcal{F}_i}=0$, $i=1,2$ and $d\theta=0$. Their transversality follows from
$TM=T\mathcal{F}_1\oplus T\mathcal{F}_2,$ and hence
$W=L_1\oplus L_2.$
Thus, a bi-Lagrangian structure naturally determines a transverse pair of Dirac structures.
Moreover, the bi-Lagrangian connection $\nabla$ induces a linear connection on $W$ that perseveres both bundles. 
The results developed here therefore provide a
framework for relating lifted bi-Lagrangian structures, induced
connections, Dirac structures, and Courant integrability on
$W=TM\oplus T^*M$.

Before explaining our results more precisely and proving them, it is necessary to present some definitions, fix some notations, and recall some known results that we shall need.
		
	\section{Notation and basic definitions}\label{sub1}
Throughout this paper, all manifolds, maps, tensor fields, and other geometry
objects are assumed to be smooth.
		
We adopt the Einstein summation convention: whenever an index appears once as
a superscript and once as a subscript in a monomial, summation over its full
range is understood. For example,
$$\lambda^i\xi_i=\sum_{i=1}^n \lambda^i\xi_i.$$
In the same way,
$$X^i\partial_i=X^i\frac{\partial}{\partial x^i}
		=\sum_{i=1}^n X^i\frac{\partial}{\partial x^i}.
$$

For  $k\in\mathbb{N}\setminus\{0\}$, we write  $[k]:=\{1,2,\dots,k\}$. 
We denote by $I_k$ the $k\times k$ identity matrix over $\mathbb{R}$.

Let $M$ be a $2n$-dimensional manifold. We write

\begin{equation*}
	\begin{aligned}
	&(x^j)=(p^i,q^i) := (p^1,\ldots,p^n,q^1,\ldots,q^n),\\
&(x^j,y^j)=(p^i,q^i,y^j) := (p^1,\ldots,p^n,q^1,\ldots,q^n,y^1,\ldots,y^{2n}), \\
&(x^j,\xi_j)=(p^i,q^i,\xi_j) := (p^1,\ldots,p^n,q^1,\ldots,q^n,\xi_1,\ldots,\xi_{2n})
\end{aligned}
\end{equation*}
for local coordinates on $M$, $TM$, and $T^*M$, respectively.

Let $\mathcal{F}$ be a foliation on a manifold $M$. We denote by  $T\mathcal{F}\subset TM$ its tangent bundle and by
$\Gamma\left(T\mathcal{F}\right)$, or simply $\Gamma\left(\mathcal{F}\right)$,
the set of smooth sections of 	$T\mathcal{F}$, respectively.
		
Let $\psi : M\longrightarrow N$ be a diffeomorphism. The push forward foliation of $\mathcal{F}$ denoted by $\psi_*\mathcal{F}=\{\psi(\mathcal{F}_x)\}_{x\in M}
$ is a foliation of $N$ whose leaves are the images under $\psi$ of the leaves  $\mathcal{F}$. Equivalently, its tangent distribution is given by
\begin{equation*}
	\Gamma\left(\psi_*\mathcal{F}\right)
:=
\{\psi_*X,\; X \in\Gamma\left(\mathcal{F}\right)\}
	=
\psi_*\Gamma\left(\mathcal{F}\right).
\end{equation*}

Let $\nabla\in\operatorname{Conn}(M)$ be a linear connection on $M$. The
torsion tensor of $\nabla$, denoted by $T_\nabla$ (or simply by $T$ when no
confusion can arise), and its curvature tensor, denoted by $R_\nabla$ (or
simply by $R$), are defined, respectively, by
$$T_\nabla(X,Y)=\nabla_XY-\nabla_YX-[X,Y],
\quad X,Y\in\mathfrak{X}(M),$$
and
$$R_\nabla(X,Y)Z=\nabla_X\nabla_YZ-\nabla_Y\nabla_XZ-\nabla_{[X,Y]}Z,
\quad X,Y,Z\in\mathfrak{X}(M).$$
Here, $[X,Y]$ denotes the Lie bracket of the vector fields $X$ and $Y$.

We say that a connection $\nabla$
\begin{enumerate}
\item[-] parallelizes a 2-form $\omega$ if $\nabla\omega=0$; that is,
\begin{equation*} 
\omega(\nabla_{X}Y,Z)+\omega(Y,\nabla_{X}Z)=X\omega(Y,Z),\quad X,Y,Z\in\mathfrak{X}(M);
\end{equation*}
			
\item[-] preserves a foliation $\mathcal{F}$ if
$\nabla\Gamma\left(\mathcal{F}\right)\subseteq
\Gamma\left(\mathcal{F}\right)$; more precisely,
\begin{equation*}  
\nabla_XY\in\Gamma\left(\mathcal{F}\right),
\quad (X,Y)\in \mathfrak{X}(M)\times\Gamma\left(\mathcal{F}\right).
\end{equation*}
\end{enumerate}
		
A Lie algebroid over a smooth manifold $M$ is a vector bundle
$A\to M$ equipped with a Lie bracket $
[\,\cdot,\cdot\,]_A:
\Gamma(A)\times\Gamma(A)\longrightarrow\Gamma(A)
$
on its space of smooth sections and a vector bundle morphism
$
\rho_A:A\longrightarrow TM,
$
called the anchor, such that
$
[a,fb]_A
=
f[a,b]_A+\bigl(\rho_A(a)f\bigr)b
$
for all $a,b\in\Gamma(A)$ and $f\in C^\infty(M)$.




		\section{Technical tools}
		
\subsection{Symplectic manifolds}

Symplectic manifolds provide natural geometric frameworks for the study of
Hamiltonian dynamics. They form a category whose objects are symplectic
manifolds and whose morphisms are symplectomorphisms. Among the fundamental
examples of symplectic manifolds, the cotangent bundle of any smooth manifold
carries a canonical symplectic form. Moreover, every diffeomorphism between two
manifolds induces a symplectomorphism between their cotangent bundles, endowed
with their respective canonical symplectic forms. This subsection is devoted to
precise formulations of these facts. 
Assume now that $M$ is endowed with a symplectic form $\omega$. Then
$\dim M=2n$ for some integer $n$. Since $\omega$ is non-degenerate, it induces
a vector bundle isomorphism $\omega^\flat:TM\longrightarrow T^*M,\quad
(x,X_x)\longmapsto \bigl(x,\omega_x(X_x,\cdot)\bigr),$
whose inverse will be denoted by
$\omega^\sharp:T^*M\longrightarrow TM$.

A symplectomorphism between two symplectic manifolds $(M_1,\omega_1)$ and
$(M_2,\omega_2)$ is a diffeomorphism
$\psi:M_1\longrightarrow M_2$
such that
$
\psi^*\omega_2=\omega_1.
$
The set of all symplectomorphisms from $(M,\omega)$ to itself is denoted by
$
\operatorname{Symp}(M,\omega),
$
and it is a group under composition.

Let 
${}^*\hspace{-0.1cm}\pi:T^*M\longrightarrow M$ be the natural projection. The tautological
$1$-form, also called the Liouville $1$-form, is the $1$-form $\Theta$ on
$T^*M$ defined by
\begin{equation*}
	\Theta_{(x,\alpha_x)}(v)
	=
	\alpha_x\left(T_{(x,\alpha_x)}{}^*\hspace{-0.1cm}\pi(v)\right),
	\quad
	(x,\alpha_x)\in T^*M,\quad v\in T_{(x,\alpha_x)}T^*M.
\end{equation*}
The exterior differential $d\Theta$ is called the canonical symplectic form,
or Liouville $2$-form, on the cotangent bundle $T^*M$.
For any coordinate chart 
$
(T^*U,x^1,\dots,x^m,\xi_1,\dots,\xi_m),
$
we have		
\begin{equation}\label{eq_Liouville_GL}
\Theta=\xi_i\,dx^i,\quad\text{and}\quad d\Theta=d\xi_i\wedge dx^i.
\end{equation}		
		
\begin{proposition}\label{liouville}
	Let $M$ be a smooth manifold. Then the cotangent bundle $T^*M$, endowed with
	the canonical symplectic form $d\Theta$, is a symplectic manifold.
\end{proposition}

\begin{proposition}\label{lifting of symplecto}
	Let $M_1$ and $M_2$ be smooth manifolds, and let
	$\varphi:M_1\longrightarrow M_2$ be a diffeomorphism. The lift
	$$
	\widehat{\varphi}:T^*M_1\longrightarrow T^*M_2,\quad (x,\alpha_x)\mapsto
	\left(
	\varphi(x),
	\alpha_x\circ T_{\varphi(x)}\varphi^{-1}
	\right).$$
	of $\varphi$ is a symplectomorphism from
	$(T^*M_1,d\Theta_1)$ to $(T^*M_2,d\Theta_2)$, where $d\Theta_1$ and
	$d\Theta_2$ are the canonical symplectic forms on $T^*M_1$ and $T^*M_2$,
	respectively.
\end{proposition}
\subsection{Bi-Lagrangian  manifolds}

In this section, we briefly recall some results concerning Hess
connections and affine bi-Lagrangian structures.

A foliation $\mathcal{F}$ on a symplectic manifold $(M,\omega)$ is said to be Lagrangian if
$
\left(\Gamma(\mathcal{F})\right)^{\perp_\omega}=\Gamma(\mathcal{F})$
where
$$\left(\Gamma(\mathcal{F})\right)^{\perp_\omega}
=
\left\{
Y\in \mathfrak{X}(M):\; \omega(X,Y)=0,\; \forall X\in
\Gamma(\mathcal{F})
\right\}.$$
Thus, for every vector field $Y\in\mathfrak{X}(M)$, one has
$$\omega(X,Y)=0 \quad \text{for all}\quad X\in \Gamma(\mathcal{F})
\quad \Longleftrightarrow \quad
Y\in\Gamma(\mathcal{F}).$$

A bi-Lagrangian structure on $M$ consists of a symplectic form $\omega$ together
with a pair $(\mathcal{F}_{1},\mathcal{F}_{2})$ of transverse Lagrangian
foliations. Consequently,
\begin{equation*}
	TM=T\mathcal{F}_{1}\oplus T\mathcal{F}_{2}.
\end{equation*}
We denote by $\mathcal{B}_l(M)$ the set of bi-Lagrangian structures on $M$.

Let $(\omega,\mathcal{F}_{1},\mathcal{F}_{2})$ be a bi-Lagrangian structure on a
$2n$-dimensional manifold $M$. Every point of $M$ has an open neighborhood $U$
which is the domain of a chart with local coordinates
$(p^1,\dots,p^{n},q^1,\dots,q^{n})$ such that
\begin{equation}\label{eq_mathcal{F}_i-local_GL}		
\Gamma(\mathcal{F}_1)_{\mid U}=\left\langle\dfrac{\partial}{\partial p^1},\dots,\dfrac{\partial}{\partial p^n}\right\rangle,
,\quad\Gamma(\mathcal{F}_2)_{\mid U}=\left\langle\dfrac{\partial}{\partial q^{1}},\dots,\dfrac{\partial}{\partial q^{n}}\right\rangle.
\end{equation}
	Such a chart, and such local coordinates, are said to be adapted to the pair of
foliations $(\mathcal{F}_{1},\mathcal{F}_{2})$.
Moreover, if in these coordinates the symplectic form is given by
\begin{equation}\label{eq_omega_Darboux_GL}
	\omega=\sum_{i=1}^{n}dq^i\wedge dp^i,
\end{equation}
then the chart, and the corresponding local coordinates, are said to be adapted
to the bi-Lagrangian structure
$(\omega,\mathcal{F}_{1},\mathcal{F}_{2})$.
		
The Hess connection associated with a bi-Lagrangian structure is defined via the following theorem.

\begin{theorem}\cite[Theor. 1]{Hess}\label{Theo1_Hess}
	Let $(M,\omega,\mathcal{F}_1,\mathcal{F}_2)$ be a bi-Lagrangian manifold. Then
	there exists a unique torsion-free connection $\nabla$ on $M$ such that $
	\nabla\omega=0$,	
	and such that $\nabla$ preserves both foliations $\mathcal{F}_1$ and
	$\mathcal{F}_2$.
	
	More precisely,
	If
	$
	X=X_1+X_2$ and 
	$Y=Y_1+Y_2$
	where $X_i,Y_i\in \Gamma(T\mathcal{F}_i)$ for $i=1,2$, then
	\begin{equation}
		\nabla_XY
		=
		D(X_1,Y_1)+[X_2,Y_1]_1
		+
		D(X_2,Y_2)+[X_1,Y_2]_2
		\label{17c}
	\end{equation}
	where $[\,\cdot,\cdot\,]_i$ denotes the $\mathcal{F}_i$-component of the Lie
	bracket with respect to the decomposition
	$
	TM=T\mathcal{F}_{1}\oplus T\mathcal{F}_{2}$.
	Here $D$ is defined by
	\begin{equation}
		i_{D(X,Y)}\omega=L_X(i_Y\omega),
		\label{18c}
	\end{equation}
	for vector fields $X,Y$ tangent to the same foliation.
\end{theorem}
We say that a bi-Lagrangian structure is affine if its
Hess connection $\nabla$ is curvature-free; equivalently, if $\nabla$ is flat.
We denote by $\mathcal{B}_{lp}(M)$ the set of affine bi-Lagrangian structures
on $M$. The set $\mathcal{B}_{lp}(M)$ is characterized  as follows.

\begin{theorem}\cite[Theor. 2]{Hess}\label{Theo2_Hess}
	Let $(\omega,\mathcal{F}_1,\mathcal{F}_2)$ be a bi-Lagrangian structure on a
	$2n$-dimensional manifold $M$, and let $\nabla$ be its Hess connection. Then the
	following assertions are equivalent:
	\begin{description}
		\item[\textnormal{a)}] The connection $\nabla$ is flat.
		\item[\textnormal{b)}] Around each point of the manifold $M$, there exists a coordinate chart
		adapted to $(\omega,\mathcal{F}_1,\mathcal{F}_2)$ in which the coefficients of
		$\nabla$ vanish.
	\end{description}
\end{theorem}
		
\subsection{Dirac structure}
For further details on this part, we refer the reader to \cite{Courant1990,Gualtieri2011,LiuWeinsteinXu1997}.
Before introducing the Courant bracket on the Whitney sum $W$,  we recall that $W$ is endowed
with a canonical fiberwise symmetric and nondegenerate bilinear form of
split signature $(m,m)$, where $m=\dim M$. It is defined by
\begin{equation}\label{split<>_LED}
\left\langle X+\alpha,\,Y+\beta\right\rangle
:=
\frac{1}{2}\bigl(\alpha(Y)+\beta(X)\bigr),
\end{equation}

for all $X,Y\in\mathfrak{X}(M)$ and $\alpha,\beta\in\Omega^1(M)$. 

This pairing identifies vectors and covectors as mutually dual isotropic components of $W$ and plays a fundamental role in the definition of Courant algebroid structures, see~\cite{Courant1990,Gualtieri2011,LiuWeinsteinXu1997}.		

A vector subbundle $L\subset W=TM\oplus T^*M$ is said to be
isotropic with respect to the canonical pairing if
$\langle e_1,e_2\rangle = 0 \quad \forall e_1,e_2\in L. $
 Equivalently,
$ L\subseteq L^\perp$, where
$L^\perp:=
\bigl\{e\in W:\langle e,\ell\rangle=0
\text{ for every }\ell\in L\bigr\}.$
It is called maximally isotropic if
$ L= L^\perp.$
Since the canonical pairing on $W$ is nondegenerate and has split
signature $(m,m)$, this condition is equivalent to
$ L$ being isotropic and satisfying $\operatorname{rank}( L)=m
={1}/{2}\operatorname{rank}(W).$
A maximally isotropic subbundle of $W$ is called a
Lagrangian subbundle.

\paragraph{The Courant bracket .}
Every section of the Whitney sum $W=TM\oplus T^*M$
can be written as $e=X+\alpha$, where $X\in\mathfrak{X}(M)$ and
$\alpha\in\Omega^1(M)$. The Courant bracket is the skew-symmetric
$\mathbb{R}$-bilinear map $[\cdot,\cdot]_C:\Gamma(W)\times\Gamma(W)\longrightarrow\Gamma(W)
$ defined by
\begin{equation}\label{Courant_Braket_LDC}
	[X+\alpha,Y+\beta]_C
:=
[X,Y]
+\mathcal L_X\beta-\mathcal L_Y\alpha
-\frac{1}{2}d\bigl(\iota_X\beta-\iota_Y\alpha\bigr),
\end{equation}
where $\mathcal L_X$ denotes the Lie derivative along $X$, $d$ is the
de Rham differential, and $\iota_X$ denotes contraction by $X$.

Unlike the Lie bracket, the Courant bracket does not satisfy the Jacobi
identity on all of $\Gamma(W)$. Its Jacobiator is an exact $1$-form.
More precisely, if
$$
\operatorname{Jac}_C(e_1,e_2,e_3)
:=
[[e_1,e_2]_C,e_3]_C
+[[e_2,e_3]_C,e_1]_C
+[[e_3,e_1]_C,e_2]_C,
$$
then
$$
\operatorname{Jac}_C(e_1,e_2,e_3)
=
d\,\mathcal{N}_C(e_1,e_2,e_3),
$$
where
$$
\mathcal{N}_C(e_1,e_2,e_3)
:=
\frac{1}{3}
\left(
\left\langle[e_1,e_2]_C,e_3\right\rangle
+
\left\langle[e_2,e_3]_C,e_1\right\rangle
+
\left\langle[e_3,e_1]_C,e_2\right\rangle
\right).
$$

\paragraph{The $H$-twisted Courant bracket.}
Let $H\in\Omega^3(M)$ be a differential $3$-form. The
$H$-twisted Courant bracket on $\Gamma(W)$ is defined by
\begin{equation}\label{H-Courant_Braket_LDC}
[X+\alpha,Y+\beta]_{C,H}
:=
[X+\alpha,Y+\beta]_C+\iota_Y\iota_XH.
\end{equation}

The bracket $[\cdot,\cdot]_{C,H}$ is skew-symmetric for every
$H\in\Omega^3(M)$. If $H$ is closed, that is, if $dH=0$, then it defines
the standard $H$-twisted exact Courant algebroid structure on
$TM\oplus T^*M$. As in the untwisted case, the $H$-twisted Courant
bracket does not satisfy the Jacobi identity on arbitrary sections;
however, its restriction to an involutive isotropic subbundle does
satisfy the Jacobi identity.

\paragraph{Dirac structures.}
A Dirac structure on $M$ is a maximally isotropic subbundle
$L\subset W=TM\oplus T^*M$ whose space of smooth sections is involutive
under the Courant bracket.

More generally, let $H\in\Omega^3(M)$ be a closed $3$-form. An
$H$-twisted Dirac structure on $M$ is a maximally isotropic subbundle
$L\subset W$ whose space of smooth sections is involutive under the
$H$-twisted Courant bracket.

Suppose now that $H$ is a locally conformally closed $3$-form, meaning
that there exists a closed $1$-form $\theta\in\Omega^1(M)$ such that
$dH=H\wedge\theta$, and $d\theta=0.$
In this setting, one can define the $(H,\theta)$-twisted Courant bracket
given in~\eqref{C(theta,H)_DC}. An $(H,\theta)$-twisted Dirac structure
on $M$ is a maximally isotropic subbundle $L\subset W$ whose space of
smooth sections is involutive under this bracket.

When $\theta=0$, the condition $dH=H\wedge\theta$ reduces to $dH=0$,
and the $(H,\theta)$-twisted theory specializes to the usual
$H$-twisted Courant Dirac theory.
 
Unless otherwise stated, by a Dirac structure we mean an
untwisted Dirac structure.

In either case, the restriction of the corresponding Courant  bracket
to $\Gamma(L)$ satisfies the Jacobi identity. Consequently, together
with the anchor given by the restriction of the natural projection
$\rho:=\operatorname{pr}_{TM}\big|_L:L\longrightarrow TM,$
the bundle $L$ carries a canonical Lie algebroid structure.

\paragraph{The induced connection on $TM\oplus T^*M$.}
Let $\nabla$ be a linear connection on $TM$. The dual connection
$\nabla^*$ on $T^*M$ is characterized by compatibility with the natural
duality pairing between $T^*M$ and $TM$. More precisely, for all
$X,Y\in\mathfrak X(M)$ and $\alpha\in\Omega^1(M)$, it is defined by
\begin{equation}\label{nabla*_LED}
(\nabla_X^*\alpha)(Y)
=
X\bigl(\alpha(Y)\bigr)-\alpha(\nabla_XY).
\end{equation}
The direct-sum connection induced by $\nabla$ and $\nabla^*$ is the
linear connection $\nabla^W$ on $W$ defined by

\begin{equation}\label{Whitney-Conne_LED}
	\nabla_X^W(Y+\alpha)
	:=
	\nabla_XY+\nabla_X^*\alpha
\end{equation}
for all $X,Y\in\mathfrak X(M)$ and $\alpha\in\Omega^1(M)$.

Moreover, $\nabla^W$ preserves the natural symmetric pairing $\langle ,\rangle$ on $W$.

\subsection{Conormal prolongations}
Let $\mathcal{D}$ be a distribution on an $m$-dimensional smooth manifold
$M$. The conormal bundle of $\mathcal{D}$ is defined as its annihilator in
$T^*M$, namely
\begin{equation}\label{conormal_bundle_LED}
\begin{aligned}
	N^*\mathcal{D}
	&=
	\left\{
	(x,\alpha_x)\in T^*M:
	\alpha_x(v)=0
	\text{ for all }v\in\mathcal{D}_x
	\right\}\\
	&=
	\bigcup_{x\in M}\{x\}\times N_x^*\mathcal{D},
\end{aligned}
\end{equation}
where
\begin{equation*}
	N_x^*\mathcal{D}
	=
	\left\{
	\alpha_x\in T_x^*M:
	\alpha_x(v)=0
	\text{ for all }v\in\mathcal{D}_x
	\right\}.
\end{equation*}	
Equivalently, the space of smooth sections of $N^*\mathcal{D}$ is
\begin{equation}\label{conormal_bundle-sect_LED}
	\Gamma(N^*\mathcal{D})
	=
	\left\{
	\alpha\in\Omega^1(M):
	\alpha(X)=0
	\text{ for all }X\in\Gamma(\mathcal{D})
	\right\}.
\end{equation}

\begin{remark}\label{Conor_Rem_ED(H)}
	If $\mathcal{D}$ is a distribution of constant rank $r$, then its
	annihilator $N^*\mathcal{D}\subset T^*M$ is a vector subbundle of rank
	$\operatorname{rank}(N^*\mathcal{D})=m-r$. We define the associated
	generalized distribution by
	\begin{equation}\label{Cornorm-N^{*t}_LED}
		N^{*t}\mathcal{D}
		:=
		\mathcal{D}\oplus N^*\mathcal{D}
		\subset TM\oplus T^*M=W,
	\end{equation}
	and denote its space of smooth sections by
	$\Gamma(N^{*t}\mathcal{D})$. In particular,
	\begin{equation*}
		\Gamma(N^{*t}\mathcal{D})
		=
		\Gamma(\mathcal{D})
		\oplus
		\Gamma(N^*\mathcal{D})
		\subset \Gamma(W).
	\end{equation*}
	Moreover, $\Gamma(N^{*t}\mathcal{D})$ is a $C^\infty(M)$-module of rank
	$$
	\operatorname{rank}\bigl(\Gamma(N^{*t}\mathcal{D})\bigr)=m=\frac{1}{2}
	\operatorname{rank}\bigl(\Gamma(W)\bigr).
	$$
\end{remark}		

For every constant rank distribution
$\mathcal{D}\subset TM$, the associated subbundle
$N^{*t}\mathcal{D}$
is Lagrangian with respect to the
natural pairing on $TM\oplus T^*M$. The precise formulation is as follows.

\begin{remark}\cite[Ex.2.4]{Gualtieri2011}\label{rem_LagW_ED}
	Let $\mathcal{D}\subset TM$ be a smooth distribution of constant rank
	$r$. Then $N^{*t}\mathcal{D}=\mathcal{D}\oplus N^*\mathcal{D}$ is  Lagrangian (maximally isotropic) with respect to the natural pairing on $TM\oplus T^*M$.
\end{remark}
\begin{proposition}\cite[p.597-599]{BCWZ}\label{Dirac_Involu(H,0)_ED}
	Let $\mathcal{D}\subset TM$ be a distribution, and let $H\in \Omega^3(M)$ be a closed 3-form. Then  $N^{*t}\mathcal D$ is an
	$H$-twisted Dirac structure if and only if $\mathcal D$ is involutive and $H|_{\wedge^3\mathcal{D}}=0$. 
\end{proposition}

\subsection{Complete, vertical and horizontal lifts}
In this subsection, we review some basic facts concerning complete,
vertical, and horizontal lifts to the tangent bundle, as well as the
vertical and horizontal lifts to the cotangent bundle. For further
details on complete and vertical lifts, we refer the reader to
\cite{YKI,YKII,YI}, while horizontal lifts are discussed in
\cite{Dom,YI}. If $\zeta$ is a geometric object on a manifold $M$, we
denote its vertical, complete, and horizontal lifts to $TM$, whenever
defined, by $\zeta^{v}$, $\zeta^{c}$, and $\zeta^{h}$, respectively.
Lifts to $T^*M$ are distinguished by an asterisk; for example,
$\zeta^{*v}$ and $\zeta^{*h}$ denote the vertical and horizontal lifts
of $\zeta$ to $T^*M$, respectively. We emphasize that horizontal lifts
depend on the choice of a linear connection on $M$.

Let $\pi:TM\to M$ be the natural projection. For
$f\in C^\infty(M)$, its vertical, horizontal, and complete lifts are defined by
$$f^v=f\circ\pi= f^h,\quad f^c=df,\quad f^{*v}=f\circ{}^*\hspace{-0.1cm}\pi= f^{*h}
$$
where $df$ is viewed as a fiberwise linear function on $TM$.

Let $\nabla$ be a linear connection on $M$ defined by $\nabla_{\partial_i}{\partial_j}=\Gamma_{ij}^k{\partial_k}$, and $\nabla^*_{\partial_i}{dx^j}=-\Gamma_{ij}^k{dx^k}$ its dual.
With respect to the induced coordinates
$(x^j,y^j)$ on $TM$ and $(x^j,\xi_j)$ on $T^*M$, one has
\begin{equation}
\begin{aligned}
X^v&=X^i\frac{\partial}{\partial y^i},\quad 
X^h=X^i\frac{\partial}{\partial x^i}-X^i\Gamma^k_{ij}y^j\frac{\partial}{\partial y^k},\\
X^c&=X^i\frac{\partial}{\partial x^i}+\frac{\partial X^i}{\partial x^j}y^j\frac{\partial}{\partial y^i},
\end{aligned}
\label{eqlift5}
\end{equation}
and
\begin{equation}\label{eq_X{*h}&alpha{*v}_GL}
X^{*h}=X^i\left(\frac{\partial}{\partial x^i}+\Gamma^k_{ij}\xi_k
\frac{\partial}{\partial \xi_j}\right),
\quad\alpha^{*v}=\alpha_i\frac{\partial}{\partial \xi_i}.
\end{equation} 

Moreover, for all $X,Y\in\mathfrak{X}(M)$, one has
\begin{align}
	&[X^c,Y^c]=[X,Y]^c, \quad [X^v,Y^v]=0, \quad [X^c,Y^v]=[X,Y]^v,\label{eqlift4}\\
	&[X^h,Y^v]=(\nabla_XY)^v, \quad [X^h,Y^h]=[X,Y]^h-\bigl(R(X,Y)^*\overset{\smile}{(y,\xi)}\bigr)^v,\label{eqlift41}\\
&[X^{*h},\alpha^{*v}]= (\nabla_X\alpha)^{*v},\quad[\alpha^{*v},\beta^{*v}]= 0,
\label{eqlift43}\\
&[X^{*h},Y^{*h}]= [X,Y]^{*h}+\bigl(\xi\circ R(X,Y)\bigr)^{*v},
\label{eqlift44}
\end{align}
where $R$ is the curvature tensor of $\nabla$, $y\in T_xM$ denotes the vector corresponding to a point $(x,y)\in TM$, and $\xi\in T_x^*M$ denotes the covector corresponding to a point $(x,\xi)\in T^*M$.  

Observe that, 
\begin{equation}\label{eq_mathfrak{X}(TM)_GL}
\mathfrak{X}(TM)=\mathfrak{X}^c(M)+\mathfrak{X}^v(M)
\end{equation}
Thus, if $\omega$ is a two-form, then $\omega^c$ can be defined as:
\begin{align}
\begin{aligned}
&\omega^c(X^c,Y^c)=(\omega(X,Y))^c,\label{eq_omega^c(X^c,Y^c)_GL}\\
&\omega^c(X^c,Y^v)=(\omega(X,Y))^v,\omega^c(X^v,Y^v)=0.
\end{aligned}
\end{align}
Moreover, 
\begin{equation}
d\omega^c=\left(d\omega\right)^c.\label{eq_domega^c_GL}\end{equation}
\begin{remark}\label{Rem_decom_rang_mathcalF^{*h,*v}_GL}
The connection $\nabla$ induces a horizontal vertical splitting 
$$T(T^*M)=(TM)^h\oplus (T^*M)^{*v}=h(TM)\oplus v(T^*M).$$ 

For a foliation $\mathcal{F}$ we define horizontal vertical extension as 

\begin{equation}\label{eq_mathcalF^{*h,*v}_GL}\widehat{\mathcal{F}}^{*h,*v}_{(x,\xi)}=\left\{X^{*h}_{(x,\xi)}+\alpha^{*v}_{(x,\xi)}:X\in T_x\mathcal{F},\ \alpha\in N_x^*\mathcal{F}\right\},\quad (x,\xi)\in T^*M.
\end{equation} 
Equivalently,
\begin{equation*}
\begin{aligned}
	\widehat{\mathcal{F}}^{*h,*v}_{(x,\xi)}
=(T\mathcal{F})^{*h}_{(x,\xi)}\oplus(N^*\mathcal{F})^{*v}_{(x,\xi)}&=
h_{(x,\xi)}(T_x\mathcal{F})
\oplus
v_{(x,\xi)}(N_x^*\mathcal{F}),\quad (x,\xi)\in T^*M\\
&=(T_x\mathcal{F})^{*h}
\oplus
(N_x^*\mathcal{F})^{*v},\quad (x,\xi)\in T^*M
\end{aligned}
\end{equation*} 
Observe that  
$\operatorname{Rank}\widehat{\mathcal{F}}^{*h,*v}=\operatorname{Rank}\mathcal{F}+ (m-\operatorname{Rank}\mathcal{F})=m=\dim M.$

 For $a,b\in\{c,h,v\}$,  
$d\in\{c,h\}$, and , $\alpha,\beta\in\mathbb{R}\setminus\{0\}$
we define the distributions $\mathcal{F}^{ab}$ and $\mathcal{F}^{v(d)}_{\alpha,\beta}$
on $TM$ by
\begin{align}
	\Gamma(\mathcal{F}^{ab})
	&=
	\left\langle X^a,\;X^b;\;X\in\Gamma(\mathcal{F})\right\rangle_
	{C^\infty(TM)}.
	\label{lift of foliation ab}\vspace{0.25cm}\\
	\Gamma\left(\mathcal{F}^{v(d)}_{\alpha,\beta}\right) &=
	\left\langle X^{v(d)}:=\alpha X^v+\beta X^d;\;X\in\Gamma(\mathcal{F})\right\rangle_
	{C^\infty(TM)}\label{lift of foliation v(d)}.
\end{align}

If $\mathcal{F}^a$ and $\mathcal{F}^b$ are transverse, then their
intersection is trivial, and hence
\begin{equation} 
\Gamma\bigl(\mathcal{F}^{ab}\bigr) = \Gamma\bigl(\mathcal{F}^a\bigr) \oplus
\Gamma\bigl(\mathcal{F}^b\bigr). 
\label{eqlift7}
\end{equation}

Unless otherwise stated, all horizontal lifts are taken with respect to a fixed flat linear connection, such as the Hess connection in the affine bi-Lagrangian setting.
\end{remark}




\section{Statements and proofs of results}

\subsection{Induced bi-Lagrangian structures on the Whitney sum}
We lift to the Whitney sum $W=TM\oplus T^*M$
the dynamics defined by the push-forward action of the
symplectomorphism group  on bi-Lagrangian structures. This construction
requires compatible lifts of both the bi-Lagrangian structures and the
symplectomorphisms. The existence of the required lift of a
symplectomorphism is established in
Proposition~\ref{lifting of symplecto}. 
It therefore remains to construct the corresponding lift of a
bi-Lagrangian structure.

\begin{lemma}\label{lem_mathcalF{*h,*v}_involutive_GL}
Let $\mathcal{F}$ be a  foliation on $M$, and let $\nabla$ be a linear connection on $TM$.
The distribution $\widehat{\mathcal{F}}^{*h,*v}$ is involutive if and only if the
following conditions hold:
\begin{enumerate}
\item $\mathcal{F}$ is involutive;
\item for every $X\in\Gamma(T\mathcal{F})$ and $\alpha\in\Gamma(N^*\mathcal{F})$, $\nabla_X\alpha\in\Gamma(N^*\mathcal{F});$
\item for every $X,Y,Z\in\Gamma(T\mathcal{F})$, $R(X,Y,Z)=0.$
\end{enumerate}
\end{lemma}

\begin{proof}
The distribution is locally generated, as a $C^\infty(T^*M)$-module, by horizontal lifts $X^{*h}$ with $X\in\Gamma(T\mathcal{F})$ and vertical lifts $\alpha^{*v}$ with $\alpha\in\Gamma(N^*\mathcal{F})$, see \eqref{eq_mathcalF^{*h,*v}_GL}. 
From \eqref{eqlift43} and \eqref{eqlift44}, the equivalence of the three statements is immediate. Indeed, 
\begin{enumerate}
\item[-] from  \eqref{eqlift44}, closure of the horizontal component of $[X^{*h},Y^{*h}]$ is equivalent to $[X,Y]\in\Gamma(T\mathcal{F}),$ which gives the first condition;
\item[-] from  the first equality of \eqref{eqlift43}, closure of the mixed bracket is equivalent to $\nabla_X\alpha\in\Gamma(N^*\mathcal{F}),$  which gives the second condition;
\item[-] for the curvature term, from  the first equality of \eqref{eqlift43}, closure requires $R(X,Y)\xi\in N_x^*\mathcal{F}$
for every $(x,\xi)\in T^*M$. Hence, for every $Z\in T_x\mathcal{F}$,
$$0=\bigl(R(X,Y)^*\xi\bigr)(Z)=-\xi\bigl(R(X,Y)Z\bigr).$$
Since $\xi\in T_x^*M$ is arbitrary, this is equivalent to $R(X,Y)Z=0.$
\end{enumerate}
The vertical--vertical brackets vanish (see \eqref{eqlift43}) and therefore impose no
additional condition.  This completes the proof of the lemma. 
\end{proof}

\begin{lemma}\label{lem_mathcalF{*h,*v}_Lagrangian_GL}
The distribution $\widehat{\mathcal{F}}^{*h,*v}$ is Lagrangian for $d\Theta$ if and only if $T_\nabla(X,Y)=0$ for every $X,Y\in T\mathcal{F}$.
\end{lemma}

\begin{proof}
Since $\operatorname{rank}\widehat{\mathcal{F}}^{*h,*v}=\dim M$ (see Remark~\ref{Rem_decom_rang_mathcalF^{*h,*v}_GL}),
it has half the dimension of $T^*M$. Therefore, it is Lagrangian if and only if it is isotropic.
	
By combining \eqref{eq_Liouville_GL} and \eqref{eq_X{*h}&alpha{*v}_GL}, for $X,Y\in\Gamma(T\mathcal{F})$ and $\alpha,\beta\in\Gamma(N^*\mathcal{F})$, we have
\begin{align}
d\Theta(X^{*h},Y^{*h})&=\xi\bigl(\nabla_XY-\nabla_YX-[X,Y]\bigr)=\xi\bigl(T(X,Y)\bigr),\label{eq_dTetha(X*h,Y*h_GL)}\\
d\Theta(X^{*h},\alpha^{*v})&=-\alpha(X)=0,\label{eq_dTetha(X*h,alpha*v_GL)}\\
d\Theta(\alpha^{*v},\beta^{*v})&=0.\label{eq_dTetha(beta*v,alpha*v_GL)}
\end{align}
Here,  identity \eqref{eq_dTetha(X*h,Y*h_GL)} is evaluated at $(x,\xi)\in T^*M$.
Identity \eqref{eq_dTetha(X*h,alpha*v_GL)} vanishes because $\alpha\in N^*\mathcal{F}$, see \eqref{conormal_bundle_LED}.
	
Consequently, the only possible obstruction to isotropy is the horizontal--horizontal term. Thus $\widehat{\mathcal{F}}^{*h,*v}$ is isotropic if and only if $\xi\bigl(T(X,Y)\bigr)=0$ for every $x\in M$, every $\xi\in T_x^*M$, and every $X_x,Y_x\in T_x\mathcal{F}$. Since the covector $\xi$ is arbitrary,
this is equivalent to $T(X,Y)=0$ for all $X,Y\in T\mathcal{F}$.

Hence $\widehat{\mathcal{F}}^{*h,*v}$ is Lagrangian for $d\Theta$ if and only if the torsion of $\nabla$ vanishes along $T\mathcal{F}\times T\mathcal{F}$.
\end{proof}

\begin{remark}\label{rem_mathcalF{*h,*v}_Lagran-foliat} By combining Lemmas~\ref{lem_mathcalF{*h,*v}_involutive_GL}\&\ref{lem_mathcalF{*h,*v}_Lagrangian_GL}, if $\nabla$ is flat, torsion-free, and preserves $T\mathcal{F}$, then
$\widehat{\mathcal{F}}^{*h,v}$ defines a Lagrangian foliation of
$(T^*M,d\Theta)$.
\end{remark}

\begin{theorem}\label{theo_cotan-lift Bi-Lag_GL}
Let $(M^{2n},\omega,\mathcal{F}_1,\mathcal{F}_2)$ be a bi-Lagrangian manifold. Suppose that $M$ admits a flat, torsion-free (locally flat) connection $\nabla$ preserving both Lagrangian foliations. Then the following statements hold.
\begin{enumerate}
\item The quadruple $\left(T^*M,d\Theta,\widehat{\mathcal{F}}_1^{*h,*v},\widehat{\mathcal{F}}_2^{*h,*v}\right)$ is a bi-Lagrangian manifold. In fact, this bi-Lagrangian structure is affine: its Hess connection is flat.
		
\item The two-form $\widetilde{\omega}=\pi^*\omega+d\Theta$ is symplectic, and $\left(T^*M,\widetilde{\omega},\widehat{\mathcal{F}}_1^{*h,*v},\widehat{\mathcal{F}}_2^{*h,*v}\right)$ is a bi-Lagrangian manifold.
		
\item If, in addition, $\nabla\omega=0,$ then $\nabla$ is the Hess connection of $(M,\omega,\mathcal{F}_1,\mathcal{F}_2)$, the original bi-Lagrangian structure is
affine, and the lifted bi-Lagrangian structure $\left(T^*M,\pi^*\omega+d\Theta,\widehat{\mathcal{F}}_1^{*h,*v},\widehat{\mathcal{F}}_2^{*h,*v}\right)$ is affine as well.
\end{enumerate}
\end{theorem}
\begin{proof}
Since $\nabla$ is a locally flat connection, from Lemmas~\ref{lem_mathcalF{*h,*v}_Lagrangian_GL}\&\ref{lem_mathcalF{*h,*v}_involutive_GL},  $\widehat{\mathcal{F}}_i^{*h,*v}$, $i=1,2$ are Lagrangian foliation. 
	
Furthermore, for every $x\in M$
$$T_xM=T_x\mathcal{F}_1\oplus T_x\mathcal{F}_2\quad \mbox{and}\quad T_x^*M=N_x^*\mathcal{F}_1\oplus N_x^*\mathcal{F}_2.$$
Consequently, $T(T^*M)=\widehat{\mathcal{F}}_1^{*h,*v}\oplus\widehat{\mathcal{F}}_2^{*h,*v}.$
Hence $\left(T^*M,d\Theta,\widehat{\mathcal{F}}_1^{*h,*v},\widehat{\mathcal{F}}_2^{*h,*v}\right)$ is bi-Lagrangian.
	
We now prove that this bi-Lagrangian structure is affine. Since $\nabla$ is flat and torsion-free, around every point of $M$ there exist $\nabla$-affine coordinates. Because $\nabla$ preserves both complementary distributions, these coordinates may be chosen adapted to the foliations, say $(p^1,\ldots,p^n,q^1,\ldots,q^n),$ such that
\eqref{eq_mathcal{F}_i-local_GL} holds.
Moreover,
\begin{equation}\label{eq_mathcalF_1{*h,*v}-loca_New_i-local_GL}
\widehat{\mathcal{F}}_1^{*h,*v}=\left\langle\frac{\partial}{\partial p^i},\frac{\partial}{\partial \xi_{n+i}}\right\rangle
\quad\mbox{and}\quad
\widehat{\mathcal{F}}_2^{*h,*v}=\left\langle\frac{\partial}{\partial q^i},\frac{\partial}{\partial \xi_i}\right\rangle.
\end{equation}
Combining identities \eqref{eq_Liouville_GL}, \eqref{eq_mathcal{F}_i-local_GL},  and \eqref{eq_mathcalF_1{*h,*v}-loca_New_i-local_GL} with Theorem~\ref{Theo2_Hess}, we conclude that $\left(T^*M,d\Theta,\widehat{\mathcal{F}}_1^{*h,*v},\widehat{\mathcal{F}}_2^{*h,*v}\right)$ is affine.
	
\paragraph{}	For~(2), it easy to see that $\widetilde\omega={}^*\hspace{-0.1cm}\pi^*\omega+d\Theta$ is symplectic.

Now, for every $X,Y\in T_x\mathcal{F}_i$, 	and for every $\alpha,\beta\in N_x^*\mathcal{F}_i$, we have
\begin{equation}\label{eq_widetilde-omega()_Gl}
\begin{aligned}
\widetilde\omega\bigl(X^{*h}+\alpha^{*v},Y^{*h}+\beta^{*v}\bigr) =\omega(X,Y)+d\Theta\bigl(X^{*h}+\alpha^{*v},Y^{*h}+\beta^{*v}\bigr).
\end{aligned}
\end{equation}
The first term of \eqref{eq_widetilde-omega()_Gl} vanishes because $\mathcal{F}_i$ is Lagrangian for $\omega$, and the second vanishes because $\widehat{\mathcal{F}}_i^{*h,*v}$ is Lagrangian for $d\Theta$. Hence each $\widehat{\mathcal{F}}_i^{*h,*v}$ is Lagrangian for $\widetilde\omega$. Their involutivity and transversality have already been established in part~1. of Theorem~\ref{theo_cotan-lift Bi-Lag_GL}. It follows that $\left(T^*M,\widetilde\omega,\widehat{\mathcal{F}}_1^{*h,*v},\widehat{\mathcal{F}}_2^{*h,*v}\right)$ is bi-Lagrangian.
	
Finally, suppose that $\nabla\omega=0$. By assumption, $\nabla$ is torsion-free and preserves both Lagrangian foliations. Thus it is a symplectic connection preserving $\Gamma(\mathcal{F}_1)$ and $\Gamma(\mathcal{F}_2)$. By uniqueness of the Hess connection of a bi-Lagrangian manifold, $\nabla$ is the Hess connection of
$(M,\omega,\mathcal{F}_1,\mathcal{F}_2)$. Since $\nabla$ is flat, the original bi-Lagrangian structure is affine.
Moreover, in the adapted $\nabla$-affine coordinates used above, the condition $\nabla\omega=0$ implies that the coefficients of $\omega$ are constant. Since both ${}^*\hspace{-0.1cm}\pi^*\omega$ and $d\Theta$ then have constant coefficients, so does $\widetilde\omega={}^*\hspace{-0.1cm}\pi^*\omega+d\Theta$.
The flat coordinate connection on $T^*M$ is therefore torsion-free, preserves $\widetilde\omega$, and preserves both lifted foliations. It is consequently the Hess connection of the lifted bi-Lagrangian structure. In particular, the latter is affine. This proves part~3. of Theorem~\ref{theo_cotan-lift Bi-Lag_GL}, which completes the proof.
\end{proof}

By applying Theorem~\ref{theo_cotan-lift Bi-Lag_GL} to the case where $\nabla=\nabla^0$ is the Hess connection of the initial bi-Lagrangian structure, we obtain the following result by induction. Observe first that with the hypothesis of Theorem~\ref{theo_cotan-lift Bi-Lag_GL}, $\nabla=\nabla^0\Longleftrightarrow \nabla\omega=0$. 

\begin{corollary}\label{cor_infinite-lift_GL} 
Let $(M_0,\omega_0,\mathcal{F}_{1,0},\mathcal{F}_{2,0})=(M,\omega,\mathcal{F}_1,\mathcal{F}_2)$ be an affine bi-Lagrangian manifold, and let $\nabla^0$ be its Hess connection. Thus
$$T_{\nabla^0}=0,\quad R_{\nabla^0}=0,\quad\nabla^0\omega_0=0,\quad\nabla^0\Gamma(\mathcal{F}_{i,0})\subseteq T\mathcal{F}_{i,0},\quad i=1,2.$$
Define recursively $M_{k+1}=T^*M_k,$ and let ${}^*\hspace{-0.1cm}\pi_k:T^*M_k\to M_k$ and $d\Theta_k$ denote, respectively, the cotangent projection and the tautological one-form. Set $\omega_{k+1}=\pi_k^*\omega_k+d\Theta_k$ and $\Gamma(T\mathcal{F}_{i,k+1})=\widehat{\mathcal{F}}^{*h,*v}_{i,k}$, $i=1,2,$
where $\nabla^k$ is the Hess connection of $(M_k,\omega_k,\mathcal{F}_{1,k},\mathcal{F}_{2,k}).$ Then, for every $k\in\mathbb N$, the quadruple
$\operatorname{Lift}^k(M,\omega,\mathcal{F}_1,\mathcal{F}_2)=(M_k,\omega_k,\mathcal{F}_{1,k},\mathcal{F}_{2,k})$ is a well-defined affine bi-Lagrangian manifold. In particular, every affine bi-Lagrangian manifold is infinitely liftable.
\end{corollary}

We now extend the results of Theorem~\ref{theo_cotan-lift Bi-Lag_GL} to the Whitney sum $W=TM\oplus T^*M$. The first follows along lines similar to those of Theorem~\ref{theo_cotan-lift Bi-Lag_GL}, while the second is obtained directly by combining Lemma~\ref{Bilem1} with Theorem~\ref{theo_cotan-lift Bi-Lag_GL}.
\begin{corollary}\label{cor_whitney-sum_GL}
Let $(M^{2n},\omega,\mathcal{F}_1,\mathcal{F}_2)$ be an affine bi-Lagrangian manifold, and let $\nabla$ be its flat Hess connection. Consider the Whitney sum
$W=TM\oplus T^*M$ with natural projection $\rho:W\longrightarrow M.$
The connection $\nabla$ gives a decomposition $TW=h(TM)\oplus v(TM)\oplus v(T^*M).$
For $i=1,2$, define $\mathcal K_i=h(T\mathcal{F}_i)\oplus v(T\mathcal{F}_i)\oplus v(N^*\mathcal{F}_i).$

Let $\sigma_x$ denote the canonical symplectic form on the fibre $T_xM\oplus T_x^*M$, defined by
\begin{equation}\label{eq_sigma_GL}
\sigma_x: T_xM\oplus T_x^*M\to \mathbb{R},\quad\bigl((X,\alpha),(Y,\beta)\bigr)\mapsto\beta(X)-\alpha(Y)
\end{equation} 
Using the connection-induced splitting of $TW$, extend this fibrewise form to a two-form $\sigma^\nabla=\sigma$ on $W$ by requiring it to vanish whenever at least one of its arguments is horizontal.
 
Let $\omega_W=\rho^*\omega+\sigma.$
Then $(W,\omega_W,\mathcal K_1,\mathcal K_2)$ is an affine bi-Lagrangian manifold.
\end{corollary}

\begin{proof}
The connection $\nabla$ induces the splitting $TW=h(TM)\oplus v(TM)\oplus v(T^*M)$.
Thus, at a point $(x,u,\xi)\in W$, one has the identification
$$T_{(x,u,\xi)}W \simeq T_xM\oplus T_xM\oplus T_x^*M.$$
Under this identification, the two-form $\omega_W$ is given by
\begin{equation}
\label{eq_Omega_WDef_GL}
\begin{aligned}
{\omega_W}_{(x,u,\xi)}\bigl((X,U,\alpha),(Y,V,\beta)\bigr)&=\omega_x(X,Y)+\sigma_x\bigl((U,\alpha),(V,\beta)\bigr)\\
&=\omega_x(X,Y)+\beta(U)-\alpha(V).
\end{aligned}
\end{equation}
From \eqref{eq_Omega_WDef_GL}, it  easy to show that  $\omega_W$ is nondegenerate.

To prove closedness and the remaining assertions, we use the flatness of $\nabla$. Observe first that, since $\omega$ is closed and the pullback commutes with the exterior derivative, $\omega_W$ is closed if and only if $\sigma$ is closed.

From Theorem~\ref{Theo2_Hess}, around each point of $M$, there is $(x^i)=(p^i,q^i)$ such that \eqref{eq_mathcal{F}_i-local_GL} and \eqref{eq_omega_Darboux_GL} hold.

Let $(x^i,u^{i})$ be the induced fiber coordinates on $TM$, and let $(x^i,\xi_i)$ be the induced fiber coordinates on $T^*M$. Thus a point of $W$ is locally written as $(x^i,u^{i},\xi_i)$.

From \eqref{eqlift5}, since the chosen coordinates are $\nabla$-affine (i.e., their Christoffel symbols vanish), the horizontal vectors are spanned by
$\left\{\partial p^i,\;\partial q^i,\;i\in[n]\right\}.$

From \eqref{eq_sigma_GL}, $\sigma$ has the local expression 
\begin{equation}\label{eq_sigma-local_GF}\sigma=du^i\wedge d\xi_i+du^{n+i}\wedge d\xi_{n+i}=d(u^id\xi_i+u^{n+i} d\xi_{n+i}).\end{equation}
Observe that, because the transition matrix for the coordinates $(u^i)$ is the
inverse (transpose) of that for the coordinates $(\xi_i)$, the expression
$u^i\,d\xi_i+u^{n+i}\,d\xi_{n+i}$ defines a $1$-form on $W$, independent of the
choice of coordinates. Therefore, $\sigma$ is closed and $\omega_W$  is a symplectic form.

By combining \eqref{eq_omega_Darboux_GL}, \eqref{eq_sigma_GL} and \eqref{eq_Omega_WDef_GL}, we get
\begin{equation}\label{eq_Omega_W_GF}
\omega_W=\sum_{i=1}^{n}\left(dq^i\wedge dp^i+du^i\wedge d\xi_i+du^{n+i}\wedge d\xi_{n+i}\right).
\end{equation}
Furthermore, from \eqref{eq_mathcal{F}_i-local_GL} and \eqref{eqlift5}, we obtain
\begin{equation}\label{eq_mathcal_W_GF}
\mathcal K_1=\left\langle\frac{\partial}{\partial p^i},\frac{\partial}{\partial u^i},\frac{\partial}{\partial \xi_{n+i}}\right\rangle, 
\quad
\mathcal K_2=\left\langle\frac{\partial}{\partial q^i},\frac{\partial}{\partial u^{n+i}},\frac{\partial}{\partial \xi_i}\right\rangle.
\end{equation}
It follows immediately from these local descriptions that both $\mathcal K_1$
and $\mathcal K_2$ are involutive and Lagrangian with respect to $\omega_W$.
They are transverse and satisfy $TW=\mathcal K_1\oplus\mathcal K_2$. Hence
$(W,\omega_W,\mathcal K_1,\mathcal K_2)$ is a bi-Lagrangian manifold. By
combining \eqref{eq_Omega_W_GF} and \eqref{eq_mathcal_W_GF} together with
Theorem~\ref{Theo2_Hess}, we conclude that $(W,\omega_W,\mathcal
K_1,\mathcal K_2)$ is affine. This completes the proof of the corollary.
\end{proof}
\begin{corollary}
From Theorem~\ref{theo_cotan-lift Bi-Lag_GL}, Corollary~\ref{cor_whitney-sum_GL}, Lemma~\ref{Bilem1}, and with the notation of Theorem~\ref{theo_cotan-lift Bi-Lag_GL}, the triples $$\left((\omega^\flat)^*\omega^{{}^*\hspace{-0.1cm}\pi},
\omega^\sharp_*\widehat{\mathcal{F}}_1^{*h,*v},
\omega^\sharp_*\widehat{\mathcal{F}}_2^{*h,*v}\right)$$
and
$$
\left(\omega+(\omega^\flat)^*\omega^{{}^*\hspace{-0.1cm}\pi}+\omega^{{}^*\hspace{-0.1cm}\pi},\,\mathcal{F}_1+\omega^\sharp_*\widehat{\mathcal{F}}_1^{*h,*v}+\widehat{\mathcal{F}}_1^{*h,*v},\,\mathcal{F}_2+\omega^\sharp_*\widehat{\mathcal{F}}_2^{*h,*v}+\widehat{\mathcal{F}}_2^{*h,*v}\right)$$
define bi-Lagrangian structures on $TM$ and $W=TM\oplus T^*M$, respectively. Here $\omega^{{}^*\hspace{-0.1cm}\pi}=d\Theta$ or $\omega^{{}^*\hspace{-0.1cm}\pi}={}^*\hspace{-0.1cm}\pi^*\omega+d\Theta$.
Moreover, if $\omega^{{}^*\hspace{-0.1cm}\pi}=d\Theta$, then both structures are affine, and if $\omega^{{}^*\hspace{-0.1cm}\pi}={}^*\hspace{-0.1cm}\pi^*\omega+d\Theta$, both structures are affine when the original structure is affine.
\end{corollary}

Observe that every symplectomorphism of $(M,\omega)$ admits a lift to a symplectomorphism of $(T^*M,\omega^{{}^*\hspace{-0.1cm}\pi})$, as follows from Proposition~\ref{lifting of symplecto}. Moreover, by Theorem~\ref{theo_cotan-lift Bi-Lag_GL}, every bi-Lagrangian structure $(\omega,\mathcal{F}_1,\mathcal{F}_2)$ on $M$ induces a bi-Lagrangian structure $(\omega^{{}^*\hspace{-0.1cm}\pi},N^{*}\mathcal{F}_1,N^{*}\mathcal{F}_2)$ on $T^*M $. It is therefore natural to investigate the compatibility between this lifting
procedure and the action $\triangleright$. Combining Theorem~\ref{theo_cotan-lift Bi-Lag_GL} (more precisely, Corollary~\ref{cor_infinite-lift_GL}) with Theorem~\ref{Bitheo2}, we obtain the following result.
		
\begin{corollary}\label{cor2}
The action $\triangleright$ defined in Theorem~\ref{theo_cotan-lift Bi-Lag_GL} can be lifted infinitely in a similar sense as in Corollary~\ref{cor_infinite-lift_GL}.\end{corollary}

\begin{theorem}\label{theo_tan-lift Bi-Lag_GL}
Let $(M,\omega,\mathcal{F}_1,\mathcal{F}_2)$ be a bi-Lagrangian manifold, let $\nabla$ denote its Hess connection.
	
\begin{enumerate}
\item The quadruple $(TM,\omega^c,\mathcal{F}_1^{cv},\mathcal{F}_2^{cv})$ is a bi-Lagrangian manifold.
\item Suppose, in addition, that  $(M,\omega,\mathcal{F}_1,\mathcal{F}_2)$ is affine bi-Lagrangian. Then $		(TM,\omega^c,\mathcal{F}_1^{hv},\mathcal{F}_2^{hv})$ is an affine bi-Lagrangian manifold.
\item Let $\alpha,\beta,\mu,\nu\in\mathbb R\setminus\{0\}$ and define the distributions $\mathcal{V}$, $\mathcal{H}$, and $\mathcal{G}_{\alpha,\beta,\mu,\nu}$ on $TM$ by
\begin{align}
&\Gamma(\mathcal{V})=\Gamma(\mathcal{F}_1^v)\oplus\Gamma(\mathcal{F}_2^v),\\
&\Gamma(\mathcal{H})=\Gamma(\mathcal{F}_1^h)\oplus\Gamma(\mathcal{F}_2^h),\\
&\Gamma(\mathcal{G}_{\alpha,\beta,\mu,\nu})=\Gamma\bigl((\mathcal{F}_1)_{\alpha,\beta}^{v(h)}\bigr)\oplus\Gamma\bigl((\mathcal{F}_2)_{\mu,\nu}^{v(h)}\bigr).
\end{align}

If $\nabla$ is flat, then $\mathcal{W}_{\alpha,\beta,\mu,\nu}=\left\{\mathcal{V},\mathcal{H},\mathcal{G}_{\alpha,\beta,\mu,\nu}\right\}$ is a $3$-web on $TM$. Thus,
$\left\{\mathcal{W}_{\alpha,\beta,\mu,\nu}\;\middle|\;\alpha\beta\mu\nu\neq0\right\}$ is a family of $3$-webs on $TM$.
	\end{enumerate}
\end{theorem}

\begin{proof} Since $\omega$ is a symplectic form, it follows from \eqref{eq_mathfrak{X}(TM)_GL} that the two-form $\omega^c$ is non-degenerate (see \eqref{eq_omega^c(X^c,Y^c)_GL}) and closed (see	\eqref{eq_domega^c_GL}); hence $\omega^c$ is a symplectic form.

Since $\mathcal{F}_i$ is Lagrangian for each $i=1,2$, it follows from
\eqref{eq_omega^c(X^c,Y^c)_GL} that $\mathcal{F}_i^{cv}$ is isotropic
with respect to $\omega^c$, and from \eqref{eqlift7} that
$2\operatorname{rank}\mathcal{F}_i^{cv}=\dim TM$. Thus,
$\mathcal{F}_i^{cv}$ is Lagrangian. Moreover, it follows from
\eqref{eqlift4} that $\mathcal{F}_i^{cv}$ is involutive. Since
$TM=T\mathcal{F}_1\oplus T\mathcal{F}_2$, it follows from
\eqref{eqlift5} that the lifted distributions $\mathcal{F}_1^{cv}$
and $\mathcal{F}_2^{cv}$ are transverse. This proves the first
assertion.
	
Suppose now that $\nabla$ is flat. Then from  there exist local adapted coordinates $(p^1,\ldots,p^n,q^1,\ldots,q^n)$ such that
\eqref{eq_mathcal{F}_i-local_GL} and \eqref{eq_omega_Darboux_GL} 
Let $(y^i)$ be the corresponding fibre coordinates on $TM$. Then
\begin{align}\label{eq_mathcal{F}_1^{hv}_GL}
\mathcal{F}_1^{hv}&=\left\langle\frac{\partial}{\partial p^i},\frac{\partial}{\partial y^i}\right\rangle,\quad
\mathcal{F}_2^{hv}=\left\langle\frac{\partial}{\partial q^i},\frac{\partial}{\partial y^{n+i}}\right\rangle.
	\end{align}
	These distributions are involutive and transverse.

From \eqref{eq_omega_Darboux_GL}, we may write locally
\begin{equation}\label{eq_omega^c-Loca_GL}
\omega^c=\left(dy^i\wedge dq^j+dp^i\wedge dw^j\right).
\end{equation}
By combining \eqref{eq_mathcal{F}_1^{hv}_GL}, \eqref{eq_omega^c-Loca_GL},
and Theorem~\ref{Theo2_Hess}, the second assertion follows.
	
For the last assertion, in the same adapted affine coordinates, we have
\begin{align*}
\mathcal{V}=\left\langle\frac{\partial}{\partial y^i},\frac{\partial}{\partial y^{n+i}}\right\rangle,\quad
\mathcal{H}=\left\langle\frac{\partial}{\partial p^i},\frac{\partial}{\partial q^i}\right\rangle,\\
\mathcal{G}_{\alpha,\beta,\mu,\nu}=\left\langle\alpha\frac{\partial}{\partial y^i}+\beta\frac{\partial}{\partial p^i},\;\mu\frac{\partial}{\partial y^{n+i}}+\nu\frac{\partial}{\partial q^i}\right\rangle.
	\end{align*}
	For example, 
\begin{align*}
\mathcal{G}_{2,6,1,1}=\left\langle2\frac{\partial}{\partial y^i}+6\frac{\partial}{\partial p^i},\;1\frac{\partial}{\partial y^{n+i}}+1\frac{\partial}{\partial q^i}\right\rangle.
\end{align*}	
All the displayed generators have constant coefficients and commute. Therefore, the three distributions are involutive. Moreover, each of the three distributions has rank $2n$, which is half the dimension of $TM$. They are therefore pairwise complementary involutive distributions. Consequently, $\mathcal{W}_{\alpha,\beta,\mu,\nu}=\left\{\mathcal{V},\mathcal{H},\mathcal{G}_{\alpha,\beta,\mu,\nu}\right\}$	defines a $3$-web on $TM$.
\end{proof}	
		
\begin{corollary}
From Theorem~\ref{theo_tan-lift Bi-Lag_GL}, Lemma~\ref{Bilem1}, and with the notation and hypothesis of Theorem~\ref{theo_tan-lift Bi-Lag_GL}, the triples $$\left((\omega^\sharp)^*\omega^c,
\omega^\flat_*\mathcal{F}_1^{ab},
\omega^\flat_*\mathcal{F}_2^{ab}\right)$$
and
$$
\left(\omega+(\omega^\sharp)^*\omega^c+\omega^c,\,\mathcal{F}_1+\omega^\flat_*\mathcal{F}_1^{ab}+\mathcal{F}_1^{ab},\,\mathcal{F}_2+\omega^\flat_*\mathcal{F}_2^{ab}+\mathcal{F}_2^{ab}\right)$$
define bi-Lagrangian structures on $T^*M$ and $W=TM\oplus T^*M$, respectively. Here $ab\in\{cv,hv\}$. 

Moreover, $\omega^\flat_*\mathcal{W}_{\alpha,\beta,\mu,\nu}=\left\{\omega^\flat_*\mathcal{V},\;\omega^\flat_*\mathcal{H},\;\omega^\flat_*\mathcal{G}_{\alpha,\beta,\mu,\nu}\right\}$ forms a 3- webs on $T^*M$
\end{corollary}		

We are going to define a 3-web on $W$ from lifted foliations.
	
\begin{proposition}\label{prop_3-web-on-W_GL}
Let $(M^{2n},\omega,\mathcal{F}_1,\mathcal{F}_2)$ be a bi-Lagrangian manifold.
Suppose that $E_1,E_2,E_3\subset T^*M$ are $\nabla^*$-parallel rank-$n$ subbundles satisfying $T^*M=E_i\oplus E_j$ whenever $i\neq j$.

Let $\alpha,\beta,\mu,\nu\in\mathbb R\setminus\{0\}$. Define distributions on $W$ by
\begin{equation}\label{eq_mathcal{A}_i-on_W_GL}
\begin{aligned}
	\mathcal{A}_1&=\mathcal{V}\oplus(E_1)^{*v}=(T\mathcal{F}_1)^{v}\oplus(T\mathcal{F}_2)^{v}\oplus(E_1)^{*v},\\
\mathcal{A}_2&=\mathcal{H}\oplus(E_2)^{*v}=\nabla(T\mathcal{F}_1)^{h}\oplus \nabla(T\mathcal{F}_2)^{h}\oplus(E_2)^{*v},\\
{\mathcal{A}_3}_{\alpha,\beta,\mu,\nu}&=\Gamma(\mathcal{G}_{\alpha,\beta,\mu,\nu})\oplus (E_3)^{*v}=(T\mathcal{F}_1)_{\alpha,\beta}^{v(h)}\oplus
(T\mathcal{F}_2)_{\mu,\nu}^{v(h)}\oplus (E_3)^{*v}.
\end{aligned}
	\end{equation}
Then $\left(\mathcal{A}_1,\mathcal{A}_2,{\mathcal{A}_3}_{\alpha,\beta,\mu,\nu}\right)$
defines a 3-web on $W$.
\end{proposition}

\begin{proof}
By combining \eqref{lift of foliation ab}, \eqref{lift of foliation v(d)} and \eqref{eq_mathcal{A}_i-on_W_GL}, we get
$$\operatorname{rank}\mathcal{A}_1=\operatorname{rank}\mathcal{A}_2=	\operatorname{rank}{\mathcal{A}_3}_{\alpha,\beta,\mu,\nu}=3n=\frac12\dim W.$$

The  pairwise complementarity follows directly from \eqref{lift of foliation ab}, \eqref{lift of foliation v(d)} and \eqref{eq_mathcal{A}_i-on_W_GL}.

Since $T^*M=E_i\oplus E_j$	implies, then $(T^*M)^{*v}=(E_i)^{*v}\oplus (E_j)^{*v}$
for every $i\neq j$. Combining the tangent and cotangent-vertical decompositions gives $TW=\mathcal{A}_i\oplus\mathcal{A}_j$ for every two distinct distributions among $\mathcal{A}_1$, $\mathcal{A}_2$, ${\mathcal{A}_3}_{\alpha,\beta,\mu,\nu}$.
	
From the second identity in \eqref{eqlift44}, each $(E_i)^{*v}$ is involutive. 
In addition, $\mathcal{V}$, $\mathcal{H}$, and $\mathcal{\mathcal{G}_{\alpha,\beta,\mu,\nu}}$  are involutive (see Theorem~\ref{theo_tan-lift Bi-Lag_GL}). By combining these with the first identity in  \eqref{eqlift44}, and the fact that  each $(E_i)^{*v}$ is $\nabla^*$-parallel, we obtain that  each of
$\mathcal{A}_1$, $\mathcal{A}_2$, ${\mathcal{A}_3}_{\alpha,\beta,\mu,\nu}$ is involutive.
	
Hence these three distributions have rank $3n$, are pairwise complementary, and are involutive. They therefore define a 3-web on $W$.
\end{proof}

\paragraph{Example of the $E_i$ Appearing in Proposition~\ref{prop_3-web-on-W_GL}}
Let $(M^{2n},\omega,\mathcal{F}_1,\mathcal{F}_2)$ a bi-Lagrangian manifold. Set
$E_1=N^*\mathcal{F}_1$, $E_2=N^*\mathcal{F}_2$. Suppose that there exists a $\nabla^*$-parallel bundle isomorphism $A:E_1\longrightarrow E_2$.

For $\mu,\nu\in\mathbb R\setminus\{0\}$, define ${E_3}_{\mu,\nu}=\left\{\nu\xi+\mu A(\xi):\xi\in E_1\right\}$.

The distribution ${E_3}_{\mu,\nu}$ is $\nabla^*$-parallel. Indeed, for
$X\in\mathfrak X(M)$ and $\xi\in\Gamma(E_1)$, one has
\begin{align}
\nonumber	\nabla_X^*\bigl(\nu\xi+\mu A(\xi)\bigr)
	&=
	\nu\nabla_X^*\xi+\mu\nabla_X^*(A\xi)\\
	&=
\label{eq_nabla_X^*(nuxi+mu A(xi))_GL}	\nu\nabla_X^*\xi+\mu A(\nabla_X^*\xi),
\end{align}
where \eqref{eq_nabla_X^*(nuxi+mu A(xi))_GL} follows from the
$\nabla^*$-parallelism of $A$. Since
$\nabla_X^*\xi\in\Gamma(E_1)$ ($E_1$ is $\nabla^*$-parallel), the right-hand side of \eqref{eq_nabla_X^*(nuxi+mu A(xi))_GL} belongs to
$\Gamma(E_3)$. Thus ${E_3}_{\mu,\nu}$ is $\nabla^*$-parallel.

Then $E_1,E_2,{E_3}_{\mu,\nu}$ are $\nabla^*$-parallel rank-$n$ subbundles of $T^*M$, and $T^*M=E_i\oplus E_j$ whenever $i\neq j$.

\subsection[Dirac structures associated with conormal bundles]{Dirac structures associated with conormal bundles and the
	involutivity of distributions}
Let $\mathcal{D}\subset TM$ be a distribution of constant rank. By the Frobenius theorem, $\mathcal{D}$ is the tangent distribution of a foliation if and only if it is involutive with respect to the Lie bracket of vector fields. The notion of involutivity extends naturally to subbundles of a vector bundle equipped with an antisymmetric bracket. In this broader setting, however, involutivity alone does not necessarily give rise to a foliation unless the bracket is appropriately related to the Lie bracket of vector fields, typically through an anchor map. In this subsection, we prove that a distinguished class of distributions on the Whitney sum
$W=TM\oplus T^*M,$ naturally associated with $\mathcal{D}$, is involutive with respect to the $(H,\theta)$-Courant bracket  if and only if the original distribution $\mathcal{D}$ is integrable, or equivalently, defines a foliation of $M$.
More precisely, we establish a relation between
bi-Lagrangian structures and $(H,\theta)$-twister transverse Dirac structures.

\begin{proposition}\label{Dirac_Involu_ED}
	Let $\mathcal{D}\subset TM$ be a distribution, $\theta\in  \Omega^1(M)$ be a closed 1-form, and let $H\in \Omega^3(M)$ be a 3-form satisfying $dH=H\wedge \theta$. Then  $N^{*t}\mathcal D$ is an
	$(H,\theta)$-twisted Dirac structure if and only if $\mathcal D$ is involutive
	and $H|_{\wedge^3\mathcal{D}}=0.$ 
\end{proposition}

\begin{proof}
	By Remark~\ref{rem_LagW_ED}, $N^{*t}\mathcal{D}$ is maximally isotropic with
respect to the natural pairing (see \ref{split<>_LED}) on $TM\oplus T^*M$. It therefore
remains to characterize its involutivity under the $(H,\theta)$-Courant bracket.

Let $\theta\in\Omega^1(M)$, and let $H\in\Omega^3(M)$ satisfy
$dH=H\wedge\theta$ and $d\theta=0$.

Define the Lichnerowicz differential by
\begin{equation}\label{eq_dtheta_DC}d_\theta:\Omega^k(M)\longrightarrow\Omega^{k+1}(M),
	\quad
	d_\theta\eta=d\eta+\theta\wedge\eta,\end{equation}
and the condition $dH=H\wedge\theta$  becomes 
$d_\theta H=0.$

The $\theta$-twisted Lie derivative is defined by
$$
\mathcal L_X^\theta
:=
\iota_Xd_\theta+d_\theta\iota_X.
$$

For $X+\alpha,Y+\beta\in\Gamma(W)$, define the
$(H,\theta)$-twisted Dorfman bracket by

\begin{equation}\label{eq_Dorfman-Bracket_DC}
	[X+\alpha,Y+\beta]_{D(H,\theta)}
	=
	[X,Y]
	+
	\mathcal L_X^\theta\beta
	-
	\iota_Yd_\theta\alpha
	+
	\iota_Y\iota_XH.
\end{equation}

The associated skew-symmetric bracket is defined by
\begin{equation}\label{eq_C(H,theta)=FD(H,theta)_DC}
	[e_1,e_2]_{C(H,\theta)}
	:=
	\frac{1}{2}
	\left([e_1,e_2]_{D(H,\theta)}-[e_2,e_1]_{D(H,\theta)}
	\right).
\end{equation}
For $e_1=X+\alpha$ and $e_2=Y+\beta$, one obtains
\begin{align}
	&\begin{aligned}{}
		[X+\alpha,Y+\beta]_{C(H,\theta)}\nonumber
		=&[X,Y]
		+\mathcal L_X^\theta\beta
		-\mathcal L_Y^\theta\alpha\\
		&-\frac12d_\theta
		\bigl(\beta(X)-\alpha(Y)\bigr)
		+\iota_Y\iota_XH.
	\end{aligned}\\
&\begin{aligned}{}
	[X+\alpha,Y+\beta]_{C(H,\theta)}
	=&[X+\alpha,Y+\beta]_{C,H}
	\\
	&+\theta(X)\beta-\theta(Y)\alpha-\frac12
	\bigl(\beta(X)-\alpha(Y)\bigr)\theta.
\end{aligned}	\label{C(theta,H)_DC}
\end{align}
By combining equality~\eqref{C(theta,H)_DC}  and Proposition~\ref{Dirac_Involu(H,0)_ED}, we obtain the desired result. This completes the proof of the theorem.
\end{proof}

\begin{remark}We consider the notations of Theorem~\ref{Dirac_Involu_ED}.

Assume that $\theta=df$ for some function $f$.
Then $d_\theta\eta=e^{-f}d(e^f\eta)$ for every differential form $\eta$.

Consequently,
$d_\theta H=0
\Longleftrightarrow
d(e^fH)=0.$
Thus the rescaled three-form
$H'=e^fH$ is closed.

Define the conformal bundle (gauge) map
\begin{equation}\label{Phi_f_DC} 
	\Phi_f:TM\oplus T^*M\longrightarrow TM\oplus T^*M,
	\quad
	\Phi_f(X+\alpha)=X+e^f\alpha.
\end{equation}
Observe that the map $\Phi_f$ is a vector bundle isomorphism covering the identity on $M$, with inverse $\Phi_f^{-1}=\Phi_{-f}$.

Moreover,
\begin{equation}\label{eq_Phi-Pre-Brack_DC}
	\Phi_f\bigl([e_1,e_2]_{D(H,\theta)}\bigr)
	=
	[\Phi_f(e_1),\Phi_f(e_2)]_{D,H'}
\end{equation}
for every $e_1,e_2\in\Gamma(W)$, where $[\cdot,\cdot]_{D,H'}$ denotes the ordinary Dorfman bracket
twisted by the closed three-form $H'$ (see~\eqref{eq_Dorfman-Bracket_DC}).

Furthermore, the canonical pairing transforms conformally:
\begin{equation}\label{eq_Phi-Pre-<>_DC}
	\left\langle\Phi_f(X+\alpha),\Phi_f(Y+\beta)\right\rangle=e^f\langle X+\alpha,Y+\beta\rangle.
\end{equation}
By combining \eqref{eq_C(H,theta)=FD(H,theta)_DC}, \eqref{eq_Phi-Pre-Brack_DC}, and \eqref{eq_Phi-Pre-<>_DC}, it follows that the $(H,\theta)$-twisted structure is locally conformally equivalent to an ordinary Courant algebroid twisted by the closed three-form $H'=e^fH$.
\end{remark}

We are going to give another properties of a twisted Courant bracket.
\begin{proposition}\label{prop-properties-C-H-theta_GL}
Let $\theta\in\Omega^1(M)$ and $H\in\Omega^3(M)$ satisfy $d\theta=0$, $dH=H\wedge\theta$. Denote by $\rho:TM\oplus T^*M\longrightarrow TM$ the natural projection. Then the $(H,\theta)$-twisted Courant bracket has the following properties.

\begin{enumerate}
	\item It is skew-symmetric: $[e_1,e_2]_{C(H,\theta)}=		-[e_2,e_1]_{C(H,\theta)}.$
	
	\item It is compatible with the anchor: $		\rho\bigl([e_1,e_2]_{C(H,\theta)}\bigr)=[\rho(e_1),\rho(e_2)]$.
	
	\item For every $f\in C^\infty(M)$, 
	$$[e_1,fe_2]_{C(H,\theta)}=		f[e_1,e_2]_{C(H,\theta)}+\rho(e_1)(f)e_2-\langle e_1,e_2\rangle df.$$
	
	\item For $X,Y\in\mathfrak X(M)$,
	$$[X,Y]_{C(H,\theta)}=[X,Y]+\iota_Y\iota_XH.$$
	
	\item For $X\in\mathfrak X(M)$ and $\beta\in\Omega^1(M)$,
	$$[X,\beta]_{C(H,\theta)}=\mathcal L_X\beta-\frac12d\bigl(\beta(X)\bigr)+\theta(X)\beta-\frac12\beta(X)\theta.$$
	
	\item For $\alpha,\beta\in\Omega^1(M)$, $[\alpha,\beta]_{C(H,\theta)}=0.$
	
	\item The Jacobi identity does not hold in general. More precisely,
	$$\operatorname{Jac}_{C(H,\theta)}(e_1,e_2,e_3)=d_\theta\mathcal{N} _{C(H,\theta)}(e_1,e_2,e_3),$$
	where
	$$
	\begin{aligned}
		\operatorname{Jac}_{C(H,\theta)}(e_1,e_2,e_3):=\;&[[e_1,e_2]_{C(H,\theta)},e_3]_{C(H,\theta)}\\
		&+[[e_2,e_3]_{C(H,\theta)},e_1]_{C(H,\theta)}\\
		&+
		[[e_3,e_1]_{C(H,\theta)},e_2]_{C(H,\theta)}
	\end{aligned}
	$$
	and
	$$
	\mathcal{N} _{C(H,\theta)}(e_1,e_2,e_3)=\frac13\sum_{\mathrm{cyc}}\left\langle[e_1,e_2]_{C(H,\theta)},e_3\right\rangle.$$
	Here $d_\theta\mathcal{N} $ is regarded as the section $0+d_\theta\mathcal{N} $ of $TM\oplus T^*M$.
\end{enumerate}
\end{proposition}

\begin{proof}

For the third property, using
$$[e_1,e_2]_{C(H,\theta)}=[e_1,e_2]_{C,H}+\theta(X)\beta-\theta(Y)\alpha-\frac12\bigl(\beta(X)-\alpha(Y)\bigr)\theta,$$
we observe that the additional $\theta$-dependent terms are $C^\infty(M)$-linear in the second entry. Thus,
\begin{equation}\label{eq_[e_1,fe_2]_GL}
	[e_1,fe_2]_{C(H,\theta)}=f[e_1,e_2]_{C(H,\theta)}+X(f)(Y+\beta)-\frac12\bigl(\alpha(Y)+\beta(X)\bigr)df.
\end{equation}
By combining \eqref{split<>_LED} and \eqref{eq_[e_1,fe_2]_GL}, we obtain
$$[e_1,fe_2]_{C(H,\theta)}=f[e_1,e_2]_{C(H,\theta)}+\rho(e_1)(f)e_2-\langle e_1,e_2\rangle\,df.$$

Jacobiator identity: Since $d\theta=0$, around every point there is an open neighborhood $U$ and a function $f\in C^\infty(U)$ such that $\theta|_U=df.$
From \eqref {eq_Phi-Pre-Brack_DC}, the bundle isomorphism $\Phi_f$ (see \eqref{Phi_f_DC}) satisfying 
$$\Phi_f\bigl([e_1,e_2]_{C(H,\theta)}\bigr)=[\Phi_f(e_1),\Phi_f(e_2)]_{C,e^fH}.$$
Moreover, $d(e^fH)=0.$	Thus, $[\,\cdot,\cdot\,]_{C,e^fH}$ is the usual Courant bracket twisted by the closed three-form $e^fH$. Its Jacobiator satisfies
\begin{equation}\label{eq_Jac_{C,e^fH}_GL} 
	\operatorname{Jac}_{C,e^fH}=d\mathcal{N} _{C,e^fH}.
\end{equation}
From \eqref{eq_Phi-Pre-<>_DC} and \eqref{eq_Jac_{C,e^fH}_GL}, we get
$$\mathcal{N} _{C,e^fH}\bigl(\Phi_f(e_1),\Phi_f(e_2),\Phi_f(e_3)\bigr)=e^f\mathcal{N} _{C(H,\theta)}(e_1,e_2,e_3).$$
Therefore,
\begin{equation}\label{eq_Phi_fJac_GL} 
	\begin{aligned}\Phi_f\bigl(\operatorname{Jac}_{C(H,\theta)}(e_1,e_2,e_3)\bigr)&=
		d\left(e^f\mathcal{N} _{C(H,\theta)}(e_1,e_2,e_3)\right)\\
		&=e^f\left(d\mathcal{N} _{C(H,\theta)}+\theta\wedge\mathcal{N} _{C(H,\theta)}\right)\\
		&=e^f d_\theta\mathcal{N} _{C(H,\theta)}.
	\end{aligned}
\end{equation}
Since $\Phi_f$ multiplies one-form components by $e^f$, then \eqref{eq_Phi_fJac_GL} implies
$$\operatorname{Jac}_{C(H,\theta)}(e_1,e_2,e_3)=d_\theta\mathcal{N} _{C(H,\theta)}(e_1,e_2,e_3).$$
As the argument is local and the resulting identity is intrinsic, it holds globally.

The other results of the proposition come directly from the definition of $[,]_{C(H,\theta)}$.
\end{proof}

Recall that a maximally isotropic subbundle
$L\subseteq TM\oplus T^*M$ inherits a Lie algebroid structure whenever
its space of sections is closed under the relevant Courant bracket;
equivalently, whenever $L$ is a Dirac structure
(see~\cite[Prop.~3.27]{Gualtieri2011}). The Lie algebroid bracket is
the restriction of the Courant bracket to $\Gamma(L)$, and its anchor
is the restriction of the natural projection $\rho:TM\oplus T^*M\longrightarrow TM.$
The same conclusion holds in the $(H,\theta)$-twisted setting,
provided $d\theta=0$, $dH=H\wedge\theta$.
The precise formulation and verification proceeds as follows.

\begin{remark}\label{rem_Lie_Al&lag_GL}
Let $L\subset TM\oplus T^*M$ be maximally isotropic and involutive with
respect to $[\,\cdot,\cdot\,]_{C(H,\theta)}$. Since $\langle e_1,e_2\rangle=0$ for all $e_1,e_2\in\Gamma(L)$, the Leibniz identity reduces to
$$[e_1,fe_2]_{C(H,\theta)}=f[e_1,e_2]_{C(H,\theta)}+\rho(e_1)(f)e_2.$$
Moreover, involutivity and isotropy imply $$\mathcal N_{C(H,\theta)}(e_1,e_2,e_3)=0$$
for all $e_1,e_2,e_3\in\Gamma(L)$. Consequently, the Jacobiator
vanishes on $\Gamma(L)$. 
Thus
\[
\left(
L,
[\,\cdot,\cdot\,]_{C(H,\theta)}\big|_{\Gamma(L)},
\rho\big|_L
\right)
\]
is a Lie algebroid.
\end{remark}

Combining this observation (Remark~\ref{rem_Lie_Al&lag_GL}) with
Proposition~\ref{Dirac_Involu_ED}, we obtain the following result.
\begin{corollary}\label{Bi-Lag&DiracStrc_ED}
Let $\mathcal{D}\subset TM$ be a distribution of constant rank, and let
$\rho:TM\oplus T^*M\longrightarrow TM$ be the natural projection. Let $\theta\in\Omega^1(M)$ and $H\in\Omega^3(M)$ satisfy $d\theta=0$, $dH=H\wedge\theta$, and $H\big|_{\wedge^3\mathcal{D}}=0$. Then
$$\left(N^{*t}\mathcal{D},[\,\cdot,\cdot\,]_{C(H,\theta)}\big|_{\Gamma(N^{*t}\mathcal{D})},\rho\big|_{N^{*t}\mathcal{D}}\right)$$ is a Lie algebroid if and only if $\mathcal{D}$ is involutive.

Furthermore, let $(\mathcal{F}_1,\mathcal{F}_2)$ be a pair of transverse Lagrangian distributions on the symplectic manifold $(M,\omega)$. Suppose that
$$d\theta=0,\quad dH=H\wedge\theta,\quad H\big|_{\wedge^3\mathcal{F}_i}=0,\quad i=1,2.$$
Then the following assertions are equivalent:
\begin{enumerate}
	\item $\bigl(N^{*t}\mathcal{F}_1,N^{*t}\mathcal{F}_2\bigr)$ is a pair of transverse Dirac structures;		
	\item $(\mathcal{F}_1,\mathcal{F}_2)$ is a bi-Lagrangian structure on $(M,\omega)$;
	\item $\bigl(N^{*t}\mathcal{F}_1,N^{*t}\mathcal{F}_2\bigr)$ is a pair of transverse $(H,\theta)$-twisted Dirac structures.
\end{enumerate}
\end{corollary}

\begin{proof}
Under the assumptions $d\theta=0$, $dH=H\wedge\theta$, $H\big|_{\wedge^3\mathcal{D}}=0$, Proposition~\ref{Dirac_Involu_ED} shows that
$N^{*t}\mathcal{D}$ is an $(H,\theta)$-twisted Dirac structure if and
only if $\mathcal{D}$ is involutive. From 
Remarks~\ref{rem_LagW_ED}\&\ref{rem_Lie_Al&lag_GL}, the restriction of the twisted
Courant bracket therefore equips $N^{*t}\mathcal{D}$ with a Lie algebroid structure exactly when $\mathcal{D}$ is involutive.

The transversality of $\bigl(N^{*t}\mathcal{F}_1,N^{*t}\mathcal{F}_2\bigr)$
follows directly from that of $(\mathcal{F}_1,\mathcal{F}_2)$.

Finally, because $\mathcal{F}_1$ and $\mathcal{F}_2$ are already transverse Lagrangian distributions, they define a bi-Lagrangian structure if and only if both distributions are involutive. Applying Proposition~\ref{Dirac_Involu_ED} separately to $\mathcal{F}_1$ and $\mathcal{F}_2$, first in the untwisted ($H=0$ and $\theta=0$) case and then in the $(H,\theta)$-twisted case, proves the equivalence of 1., 2., and 3.
\end{proof}

Observe that, if a linear connection $\nabla$ preserves a distribution
$\mathcal{D}$, then its dual connection preserves the corresponding
conormal bundle. By combining this with Theorem~\ref{Theo1_Hess}, and Corollary~\ref{Bi-Lag&DiracStrc_ED}, we have the following result

\begin{corollary}
Let $(M,\omega,\mathcal{F}_1,\mathcal{F}_2)$ be a bi-Lagrangian manifold.
The pair of Dirac structures
$\bigl(N^{*t}(\Gamma(\mathcal{F}_1)),N^{*t}(\Gamma(\mathcal{F}_2))\bigr)$
admits a unique torsion-free linear connection $\nabla^{W}$, of the form
$(\nabla,\nabla^{*})$, on $W$, that preserves both Dirac structures and
such that $\nabla$ parallelizes $\omega$. This connection is given by 
$$\nabla_X^W(Y+\alpha)=	\nabla_XY+\nabla^*_X\alpha \quad X,Y\in\mathfrak{X}(M),\quad\alpha\in\Omega^1(M)
$$
where $\nabla$ is the Hess connection of $(M,\omega,\mathcal{F}_1,\mathcal{F}_2)$.
Moreover, the explicit expression for the induced connection
$\nabla^{W}$ follows directly from Theorem~\ref{Theo1_Hess}. Indeed, the
bi-Lagrangian splitting $TM=T\mathcal{F}_1\oplus T\mathcal{F}_2$
induces the dual decomposition
\begin{equation*}
	T^*M=(\Gamma(\mathcal{F}_1))^*\oplus(\Gamma(\mathcal{F}_2))^*
	\cong N^*(\Gamma(\mathcal{F}_2))\oplus N^*(\Gamma(\mathcal{F}_1)).
\end{equation*}
Accordingly, write $X=X_1+X_2$, 
$Y=Y_1+Y_2$, $\alpha=\alpha_1+\alpha_2$,
where $X_i,Y_i\in\Gamma(T\mathcal{F}_i)$ and
$\alpha_i\in\Gamma((T\mathcal{F}_i)^*)$. Thus,
$\alpha(Y)=\alpha_1(Y_1)+\alpha_2(Y_2).$

By the definition of $\nabla$ in \eqref{17c} and that of its dual connection $\nabla^*$ in \eqref{nabla*_LED}, we obtain
\begin{align*}
	(\nabla_X^*\alpha)(Y)
	={}&X\bigl(\alpha_1(Y_1)+\alpha_2(Y_2)\bigr)\\
	&-\alpha_1\left(D(X_1,Y_1)+[X_2,Y_1]_1\right)+\alpha_2\left(D(X_2,Y_2)+[X_1,Y_2]_2\right)
\end{align*}
where $D$ is defined in \eqref{18c}. 

Hence
\begin{equation*}
	\nabla_X^*\alpha_1+\alpha_2
	=
	\nabla_X^*\alpha_1+\nabla_X^*\alpha_2,
\end{equation*}
where
\begin{align*}
	(\nabla_X^*\alpha_1)(Y_1)
	&=
	X\bigl(\alpha_1(Y_1)\bigr)
	-\alpha_1\left(D(X_1,Y_1)+[X_2,Y_1]_1\right),\\
	(\nabla_X^*\alpha_2)(Y_2)
	&=
	X\bigl(\alpha_2(Y_2)\bigr)
	-\alpha_2\left(D(X_2,Y_2)+[X_1,Y_2]_2\right).
\end{align*}

Consequently, the induced connection on
$W=TM\oplus T^*M$ is given explicitly by
\begin{equation*}
	\begin{aligned}
		\nabla_X^W\bigl(Y+\alpha\bigr)
		=&
		D(X_1,Y_1)+[X_2,Y_1]_1+
		D(X_2,Y_2)+[X_1,Y_2]_2
		\\
		&+
		\nabla_X^*\alpha_1+\nabla_X^*\alpha_2.
	\end{aligned}
\end{equation*}
\end{corollary}

\begin{remark}
Let $\mathcal{D}\subset TM$ be a distribution, $\theta\in  \Omega^1(M)$ be a closed 1-form, and let $H\in \Omega^3(M)$ be a 3-form satisfying $dH=H\wedge \theta$.

\paragraph{If $\theta$ is exact, then $(\Omega^\bullet(M),d_\theta)$ is isomorphic to $(\Omega^\bullet(M),d)$.\\}
Suppose that $\theta=df$ for some $f\in C^\infty(M)$. Then $$d_\theta\eta=d\eta+df\wedge\eta=e^{-f}d\bigl(e^f\eta\bigr).$$
Hence the map $\phi_{f}$ (see \eqref{Phi_f_DC}) defines an isomorphism of cochain complexes.

\paragraph{Gauge equivalence of Morse-Novikov cohomologies: If 	$[\theta]_{\operatorname{dR}}=[\theta']_{\operatorname{dR}}$, then the equivalence $H_\theta^\bullet(M)\cong H_{\theta'}^\bullet(M)$  holds.\\}

Let $\theta$ and $\theta'$ be closed 1-forms such that $\theta'=\theta+df$
for some $f\in C^\infty(M)$. Then $d_{\theta'}\bigl(e^{-f}\eta\bigr)
=e^{-f}d_\theta\eta.$
Thus the map $\phi_{-f}$ (see \eqref{Phi_f_DC}) induces an isomorphism $H_\theta^\bullet(M)\cong H_{\theta'}^\bullet(M).$

\paragraph{If $[\theta]_{\operatorname{dR}}=[\theta']_{\operatorname{dR}},$ then $\bigl(W,\langle\cdot,\cdot\rangle,
	[\cdot,\cdot]_{C(H,\theta)},\rho\bigr)\equiv\bigl(W,\langle\cdot,\cdot\rangle,
	[\cdot,\cdot]_{C(H',\theta')},\rho\bigr)$  for some $H'$.\\}

Let $f\in C^\infty(M)$ be such that $\theta'=\theta+df$, and set $H'=e^{-f}H$. Here, $\Phi_f$ is given in \eqref{Phi_f_DC}. We then have
\begin{equation}\label{eq_PhiC(H,theta)=C(H',theta')_DC}
	\Phi_{-f}\bigl([e_1,e_2]_{C(H,\theta)}\bigr)
	=
	[\Phi_f(e_1),\Phi_f(e_2)]_{C(H',\theta')}
\end{equation}
for all $e_1,e_2\in W$.

Observe that 
\begin{equation}\label{eq_dH=HWtheta_equidH'=H'Wtheta'_DC}
	dH=H\wedge\theta\quad\Longleftrightarrow\quad
	dH'=H'\wedge\theta'.
\end{equation}

Combining  \eqref{eq_Phi-Pre-<>_DC}, \eqref{eq_PhiC(H,theta)=C(H',theta')_DC},  \eqref{eq_dH=HWtheta_equidH'=H'Wtheta'_DC}, we conclude that the bundle isomorphism
$\Phi_f$ defines a conformal Courant algebroid equivalence
$$\Phi_f:\bigl(W,\langle\cdot,\cdot\rangle,
[\cdot,\cdot]_{C(H,\theta)},\rho\bigr)
\longrightarrow
\bigl(W,\langle\cdot,\cdot\rangle,
[\cdot,\cdot]_{C(H',\theta')},\rho\bigr).$$

\end{remark}
\section{Appendix}
We also established a result concerning the push forward of a
bi-Lagrangian structure by a symplectomorphism. An alternative proof of
this result was given in \cite{TNB}.		
		
\begin{remark}\label{act}
	Let $(M,\omega)$ be a symplectic manifold, and let
	$\psi,\varphi:M\longrightarrow M$ be two diffeomorphisms. Recall that
	$(\psi\circ\varphi)_*=\psi_*\circ\varphi_*.$
	
	For every connection $\nabla\in \operatorname{Conn}(M)$ and every diffeomorphism
	$\psi$ of $M$, define a new connection $\nabla^\psi$ by
	$$(\nabla^\psi)_X Y	=\psi_*\left(
	\nabla_{\psi_*^{-1}X}\,\psi_*^{-1}Y
	\right)$$	for every $X,Y\in\mathfrak{X}(M)$. Then
	$\nabla^\psi\in \operatorname{Conn}(M)$.
	Moreover, for all diffeomorphisms $\psi,\varphi:M\to M$, one has
	$\nabla^{\psi\circ\varphi}=(\nabla^\varphi)^\psi.$
\end{remark}

\begin{lemma}\label{Bilem1}
	Let $(M,\omega,\mathcal{F}_1,\mathcal{F}_2)$ be a bi-Lagrangian manifold. Let $N$ be a manifold diffeomorphic to
	$M$. Then, for any diffeomorphism $\psi:M\longrightarrow N$, the 	
	triple
	$
	\left((\psi^{-1})^*\omega,\psi_*\mathcal{F}_1,\psi_*\mathcal{F}_2\right)
	$
	defines a bi-Lagrangian structure on $N$. 
	
	Moreover, if $\nabla$ denotes the Hess connection of
	$(\omega,\mathcal{F}_1,\mathcal{F}_2)$, then the Hess connection of	$\left((\psi^{-1})^*\omega,\psi_*\mathcal{F}_1,\psi_*\mathcal{F}_2\right)
	$
	is the pushed-forward connection $\nabla^\psi$. Furthermore, the structure
	$(\omega,\mathcal{F}_1,\mathcal{F}_2)$ is affine if and only if the structure $\left((\psi^{-1})^*\omega,\psi_*\mathcal{F}_1,\psi_*\mathcal{F}_2\right)$
	is affine.
\end{lemma}
\begin{proof}
	Set $\omega_N := (\psi^{-1})^*\omega$.
	Since $\psi^{-1}:N\to M$ is a diffeomorphism and $\omega$ is symplectic, the
	form $\omega_N$ is also symplectic. Indeed, it is closed because
	$$	d\omega_N
	=
	d((\psi^{-1})^*\omega)
	=
	(\psi^{-1})^*(d\omega)
	=
	0,$$
	and it is non-degenerate because pull-backs by diffeomorphisms preserve
	non-degeneracy.
	
	We now show that $\psi_*\mathcal{F}_1$ and $\psi_*\mathcal{F}_2$ are
	Lagrangian foliations for $\omega_N$. Since $\psi$ is a diffeomorphism, it
	sends foliations to foliations and preserves their dimensions. Moreover, for
	$X,Y\in\Gamma(\mathcal{F}_1)$, we have
	$$
	\omega_N(\psi_*X,\psi_*Y)
	=
	((\psi^{-1})^*\omega)(\psi_*X,\psi_*Y)
	=
	\omega(X,Y)
	=
	0.$$
	Hence $\psi_*\mathcal{F}_1$ is isotropic. Since
	$\dim(\psi_*\mathcal{F}_1)=\dim(\mathcal{F}_1)=\frac12\dim M$, it is
	Lagrangian. The same argument applies to $\psi_*\mathcal{F}_2$.
	
	Furthermore, since
	$$
	TM=\Gamma(\mathcal{F}_1)\oplus \Gamma(\mathcal{F}_2),$$
	applying $\psi_*$ gives
	$$TN=T(\psi_*\mathcal{F}_1)\oplus T(\psi_*\mathcal{F}_2).$$
	Therefore
	$$\left(\omega_N,\psi_*\mathcal{F}_1,\psi_*\mathcal{F}_2\right)$$
	is a bi-Lagrangian structure on $N$.
	
	Now, let us prove that its Hess connection is $\nabla^\psi$. Recall that the
	Hess connection of a bi-Lagrangian manifold is the unique connection which is symplectic, and preserves both Lagrangian foliations.
	
	First, $\nabla^\psi$ is torsion-free. Indeed, for all $X,Y\in\mathfrak X(N)$,
	$$
	T_{\nabla^\psi}(X,Y)
	=
	\psi_*\left(
	T_\nabla(\psi_*^{-1}X,\psi_*^{-1}Y)
	\right).$$
	Since $T_\nabla=0$, it follows that $T_{\nabla^\psi}=0$.
	
	Next, $\nabla^\psi$ preserves the symplectic form $\omega_N$. For
	$X,Y,Z\in\mathfrak X(N)$, one has
	$$(\nabla^\psi_X\omega_N)(Y,Z)=(\nabla_{\psi_*^{-1}X}\omega)
	(\psi_*^{-1}Y,\psi_*^{-1}Z).$$
	Since $\nabla\omega=0$, we obtain
	$\nabla^\psi\omega_N=0.$
	
	Moreover, $\nabla^\psi$ preserves the two pushed forward foliations. Indeed,
	if $Y\in\Gamma(\psi_*\mathcal{F}_i)$, then $			\psi_*^{-1}Y\in\Gamma(\mathcal{F}_i).$
	Since $\nabla$ preserves $\mathcal{F}_i$, we have
	$\nabla_{\psi_*^{-1}X}\psi_*^{-1}Y
	\in\Gamma(\mathcal{F}_i).$
	Applying $\psi_*$ gives
	$(\nabla^\psi)_X Y
	\in
	\Gamma(\psi_*\mathcal{F}_i).$
	Thus $\nabla^\psi$ preserves $\psi_*\mathcal{F}_1$ and
	$\psi_*\mathcal{F}_2$.
	
	Therefore $\nabla^\psi$ is torsion-free, symplectic, and preserves the two
	Lagrangian foliations. By uniqueness of the Hess connection, $\nabla^\psi$ is
	the Hess connection of $\left((\psi^{-1})^*\omega,\psi_*\mathcal{F}_1,\psi_*\mathcal{F}_2\right).$
	
	Finally, the initial bi-Lagrangian structure is affine if and only if its Hess 	connection is flat. The curvature of $\nabla^\psi$ satisfies
	$$R_{\nabla^\psi}(X,Y)Z	=\psi_*\left(	R_\nabla(\psi_*^{-1}X,\psi_*^{-1}Y)\psi_*^{-1}Z
	\right).$$
	Therefore, $R_\nabla=0$ if and only if $R_{\nabla^\psi}=0$. Consequently,
	$(\omega,\mathcal{F}_1,\mathcal{F}_2)$ is affine if and only if
	$
	\left((\psi^{-1})^*\omega,\psi_*\mathcal{F}_1,\psi_*\mathcal{F}_2\right)
	$
	is affine.
\end{proof}

\begin{theorem}\label{Bitheo2}
	Let $(M,\omega,\mathcal{F}_1,\mathcal{F}_2)$ be a bi-Lagrangian manifold.
	Then the map
	$$	\triangleright:
	\operatorname{Symp}(M,\omega)\times\mathcal B_l(M,\omega)
	\longrightarrow
	\mathcal B_l(M,\omega),\quad
	(\psi,(\mathcal{F}_1,\mathcal{F}_2))\mapsto
	(\psi_*\mathcal{F}_1,\psi_*\mathcal{F}_2)
	$$
	is a left group action.
	
	Moreover, the inclusion
	$$	\triangleright\left(
	\operatorname{Symp}(M,\omega)\times\mathcal B_{lp}(M)
	\right)
	\subset
	\mathcal B_{lp}(M)$$
	holds, where $B_{lp}(M)$ denotes the set of affine bi-Lagrangian structure on $M$.
\end{theorem}

\begin{proof}
	By Lemma~\ref{Bilem1}, the operation $\triangleright$ is well defined.
	The action properties of $\triangleright$ follow from the action properties of
	$\operatorname{Symp}(M,\omega)$ on $\mathfrak{X}(M)$.
	Furthermore, the affine condition is preserved under this action, since
	$$R_{\nabla}=0
	\quad\Longleftrightarrow\quad
	R_{\nabla^\psi}=0 .$$
	Therefore,
	$$\triangleright\left(
	\operatorname{Symp}(M,\omega)\times \mathcal B_{lp}(M)
	\right)\subseteq
	\mathcal B_{lp}(M).$$
	This proves Theorem~\ref{Bitheo2}.
\end{proof}




		
	\end{document}